\documentclass[10pt]{article}
\usepackage[top=3.5cm, bottom=3cm, left=3cm, right=2.5cm]{geometry} 

\usepackage{amssymb,amsmath}
\usepackage{amsthm}

\usepackage{amsfonts}
\usepackage{amscd}

\usepackage{subfigure}
\usepackage{graphicx}
\usepackage{epsfig}

\usepackage{esvect} 
\usepackage{enumerate}
\usepackage{enumitem} 

\usepackage{dsfont}
\usepackage{mathtools}
\DeclarePairedDelimiter{\ceil}{\lceil}{\rceil} 

\usepackage{hyperref}

\newtheorem{thm}{Theorem}[section]
\newtheorem{lem}[thm]{Lemma}

\newtheorem{defi}[thm]{Definition}
\newtheorem{rem}[thm]{Remark}
\newtheorem{assumption}[thm]{Assumption}

\newtheorem{ex}[thm]{Example}

\newcommand{\ga}{\alpha}
\newcommand{\gb}{\beta}
\newcommand{\gd}{\delta}
\newcommand{\eps}{\ensuremath{\varepsilon}}

\renewcommand{\gg}{\gamma}
\newcommand{\gk}{\kappa}
\newcommand{\gl}{\lambda}
\newcommand{\go}{\omega}
\newcommand{\gs}{\sigma}
\newcommand{\gt}{\theta}

\newcommand{\gD}{\Delta}
\newcommand{\gG}{\Gamma}

\newcommand{\gO}{\Omega}

\newcommand{\cA}{\mathcal{A}}
\newcommand{\cB}{\mathcal{B}}

\newcommand{\cD}{\mathcal{D}}

\newcommand{\cF}{\mathcal{F}}

\newcommand{\cI}{\mathcal{I}}

\newcommand{\cK}{\mathcal{K}}
\newcommand{\cL}{\mathcal{L}}

\newcommand{\cO}{\mathcal{O}}
\newcommand{\cP}{\mathcal{P}}

\newcommand{\cU}{\mathcal{U}}

\newcommand{\cX}{\mathcal{X}}
\newcommand{\cY}{\mathcal{Y}}

\newcommand{\bE}{\mathbb{E}}

\newcommand{\bN}{\mathbb{N}}

\newcommand{\bP}{\mathbb{P}}

\newcommand{\bR}{\mathbb{R}}

\newcommand{\bT}{\mathbb{T}}

\newcommand{\1}{\mathbf{1}} 

\newcommand{\LHS}{\cL_{\textit{HS}}}
\newcommand*{\lrscript}[5]{{\vphantom{#1}}_{#2}^{#3}{#1}_{#4}^{#5}}
\newcommand{\dualpair}[4]{\ensuremath{\lrscript{\langle}{#1}{}{}{} #3 ,#4 \rangle_{#2}}}
\newcommand{\be}{\begin {equation}}
\newcommand{\ee}{\end  {equation}}
\newcommand{\bee}{\begin {equation*}}
\newcommand{\eee}{\end {equation*}}
\newcommand{\ol}{\overline}

\newcommand{\floor}[1]{\lfloor #1 \rfloor}

\newcommand{\dgj}[1]{[\![ #1 ]\!]} 
\newcommand{\dga}[1]{\{\!\{ #1 \}\!\}} 
\newcommand{\abs}[1]{\left\vert#1\right\vert}

\newcommand{\KL}{{Karhunen-Lo\`{e}ve }}

\let\inf\relax \DeclareMathOperator*\inf{\vphantom{p}inf}

\usepackage{color}

\begin{document}
	
	\title{Stochastic Transport with L\'evy Noise \\ [1ex]
		\Large{Fully Discrete Numerical Approximation} }
	\author{ 
		Andreas Stein 
		\footnote{Seminar for Applied Mathematics, ETH Zürich. Email: \href{mailto:andreas.stein@sam.math.ethz.ch}{andreas.stein@sam.math.ethz.ch}
		}
		\and
		Andrea Barth 
		\footnote{Institute of Applied Analysis and Numerical Simulation/SimTech, University of Stuttgart. Email: \href{mailto:andrea.barth@mathematik.uni-stuttgart.de}{andrea.barth@mathematik.uni-stuttgart.de}
		}
	}
	
	\date{\today}

\maketitle

\begin{abstract}
Semilinear hyperbolic stochastic partial differential equations (SPDEs) find widespread applications in the natural and engineering sciences. However, the traditional Gaussian setting may prove too restrictive, as phenomena in mathematical finance, porous media, and pollution models often exhibit noise of a different nature. To capture temporal discontinuities and accommodate heavy-tailed distributions, Hilbert space-valued Lévy processes or Lévy fields are employed as driving noise terms. 
The numerical discretization of such SPDEs presents several challenges. The low regularity of the solution in space and time leads to slow convergence rates and instability in space/time discretization schemes. Furthermore, the Lévy process can take values in an infinite-dimensional Hilbert space, necessitating projections onto finite-dimensional subspaces at each discrete time point. Additionally, unbiased sampling from the resulting Lévy field may not be feasible.
In this study, we introduce a novel fully discrete approximation scheme that tackles these difficulties. Our main contribution is a discontinuous Galerkin scheme for spatial approximation, derived naturally from the weak formulation of the SPDE. We establish optimal convergence properties for this approach and combine it with a suitable time stepping scheme to prevent numerical oscillations. Furthermore, we approximate the driving noise process using truncated Karhunen-Loève expansions. This approximation yields a sum of scaled and uncorrelated one-dimensional Lévy processes, which can be simulated with controlled bias using Fourier inversion techniques.
\end{abstract}

\medskip


\medskip


\section{Introduction}

In many applications in the natural sciences and financial mathematics partial differential equations (PDEs) are utilized to model dynamics of the underlying system. Often, the dynamical systems are subject to uncertainties for instance due to noisy data, measurement errors or parameter uncertainty. A common approach to capture this behavior is to model the source of uncertainty by continuous Gaussian processes, which are analytically tractable and  straightforward to simulate. It turns out, however, that Gaussian distributions notoriously underappreciate rare events, thus heavy-tailed, discontinuous L\'evy-processes are better suited, i.e., to model stock returns, interest rate dynamics and energy forward markets (\cite{AKW10,GS04}).
Furthermore, Gaussian random objects are unfit to capture the impact of spatial and temporal discontinuities, for example in flows through fractured porous media or composite materials, see e.g. \cite{ZK04}.
However, replacing Gaussian distributions by more general random objects comes at the cost of lower regularity (both, path-wise and in a mean-square sense) and advanced sampling techniques are required.

In this article we consider semilinear first order stochastic partial equations (SPDEs) with a random source term.  The noise is modeled by a space-time L\'evy process taking values in a suitable infinite-dimensional Hilbert space $U$. Existence and uniqueness of weak solutions to this type of equations is ensured but in general no closed formulas or distributional properties are available.
Thus, we need to rely on numerical discretization schemes to estimate moments or statistics of the solution. The numerical approximation of SPDEs has been an active field of research in the last decade.
Most publications focus on second order parabolic equations, i.e. stochastic versions of the heat or Allen-Cahn equation, see for instance \cite{BL12,BLS13,D09,GM05,GSS16,JP09,JKN11,KLL15,K14} and the references therein. In this setting, L\'evy fields as driving noise of the SPDE have been investigated, among others, in \cite{BL12b,BSt18,BZ10,DHP12,PZ07}.
Results on second order hyperbolic SPDEs may be found, e.g., in \cite{ACLW16,CQS16,KLS15,PZ07,W06} and the references therein, nonlinear hyperbolic SPDEs are the subject of interest, for example, in \cite{BJ13,KR12}.

To model the dynamics in financial markets, however, it is more common to consider first order linear hyperbolic SPDEs, for example in the \textit{Heath-Jarrow-Morton model} with \textit{Musiela parametrization} for interest rate forwards, see \cite{BM94,CT07,HJM92}. Another example may be found in \cite{BB14,BK08}, where the authors motivate a stochastic framework to model energy forward markets perturbed by infinite-dimensional noise.
The underlying SPDE is a semilinear hyperbolic transport problem, where the nonlinearity stems from a no-arbitrage condition and directly depends on the volatility in the market, represented by the integrand for the infinite-dimensional noise process.
Naturally, the numerical treatment is then more involved than in the parabolic case, as we face lower regularity of the solution and the transport semigroup is not analytic. Consequently, there is very little literature on the numerical analysis of stochastic transport problems as for example \cite{B10,KT14}. More recently, in~\cite{benth2022heat} the authors derived error estimates for a finite difference approximation of stochastic transport on a one-dimensional spatial domain.

Our contribution is a rigorous regularity analysis and a fully discrete approximation scheme for a stochastic transport equation on a one- and two-dimensional spatial domain, and driven by a trace class L\'evy noise $L$. We derive mean-square temporal continuity and spatial regularity of the solution in terms of fractional Sobolev norms under mild assumptions. The degree of spatial smoothness depends on the regularity of $L$ and is made explicit and outlined in detail for the important special case that $L$ is associated to a Mat\'ern covariance function (see Examples~\ref{ex:matern} and ~\ref{ex:matern2} below). Furthermore, we consider the transport problem on a bounded domain with suitable inflow boundary conditions rather than on $\bR^d$. This is of practical interest in terms of modeling and simulation, but the boundary naturally limits spatial regularity of the solution even for smooth noise and initial conditions.

To approximate the solution, we couple a stable time stepping scheme with a discontinuous Galerkin (DG) approach for the spatial domain, exploiting the weak formulation of the SPDE. 
By imposing certain flow conditions on the mesh as in \cite{CDG08} (see Assumption~\ref{ass:flow} in Section~\ref{sec:dg} for details), we guarantee optimal spatial convergence. 
This method has been proven to be successful for deterministic hyperbolic problems, but, to the best of our knowledge, has not been applied for the discretization of SPDEs.
Finally, to sample the paths of $L$ and to obtain a fully discrete scheme, we combine truncated \KL expansions with an arbitrary approximation algorithm for the one-dimensional marginal L\'evy processes. In each step we provide bounds on the strong mean-squared error and give an estimate of the overall error between the unbiased solution and its fully discrete numerical approximation.

In Section~\ref{sec:prelim}, we present a comprehensive introduction to SPDEs with Lévy noise, establishing the existence and uniqueness of mild/weak solutions in a general setting. Moving on to Section~\ref{sec:ste}, we introduce the stochastic transport problem associated with a first-order differential operator and outline the necessary assumptions for ensuring well-posedness. Subsequently, we demonstrate the spatial Sobolev regularity and mean-square temporal regularity of the solution in order to establish a rigorous error control framework for subsequent sections.
In Section~\ref{sec:dg}, we introduce a discontinuous Galerkin spatial discretization method, which is then combined with a backward Euler time stepping scheme in Section~\ref{sec:time}. By utilizing a DG mesh that respects the flow direction of the transport operator, we derive optimal convergence rates for this spatio-temporal discretization approach.
Section~\ref{sec:noise} focuses on the sampling procedure of the infinite-dimensional driving noise. We provide an overall mean-squared error analysis that encompasses temporal, spatial, and noise approximation components. Finally, in Section~\ref{sec:numerics}, we present a numerical example that serves to support our theoretical findings.

\section{Stochastic Partial Differential Equations with L\'evy Noise}\label{sec:prelim}

Let $(\gO,\cF,\bP, (\cF_t,t\ge0))$ be a filtered probability space satisfying the usual conditions and let $\bT=[0,T]$ be a finite time interval.
Furthermore, let $(U,(\cdot,\cdot)_U)$ and $(H,(\cdot,\cdot)_H)$ be separable Hilbert spaces and let $\cL(U, H)$ and $\cL(H)$ denote the set of linear bounded operators $O: U\to  H$ and $O: H\to  H$, respectively.
The space of \textit{Hilbert-Schmidt operators} on $U$ is given by
$$\LHS(U,H):=\left\{O\in \cL( U,  H)\big|\; \|O\|^2_{\LHS(U; H)}:=\sum_{k\in\bN} \|O u_k\|^2_{ H}<+\infty \right\},$$
where $(u_k,k\in\bN)$ is an arbitrary orthonormal basis of $U$.
The Lebesgue-Bochner space of all square-integrable, $ H$-valued random variables is defined as
\bee
L^2(\gO; H):=\left\{Y:\gO\to H \big|\,\text{$Y$ is strongly measurable with $\|Y\|_{L^2(\gO; H)}:=\bE(\|Y\|_{ H}^2)^{1/2}<+\infty$}\right\}.
\eee
For the remainder of this article, we omit the stochastic argument $\go\in\gO$ for notational convenience.
Solutions to the SPDEs are characterized by path-wise identities that hold almost surely, see Definition~\ref{def:solutions} below. Therefore, unless stated otherwise, all appearing equalities and estimates involving stochastic terms are in the path-wise sense and are assumed to hold almost surely.
We denote by $C$ a generic positive constant which may change from one line to another.
Whenever necessary, the dependency of $C$ on certain parameters is made explicit.
Our focus is on stochastic partial differential equations with L\'evy noise, i.e. with a possibly infinite-dimensional, square-integrable L\'evy process as driving noise.
\begin{defi}
	A $U$-valued stochastic process $L=(L(t),t\in\bT)$ is called \textit{L\'evy process} if
	\begin{itemize}
		\item $L$ has stationary and independent increments,
		\item $L(0)=0$ almost surely, and
		\item $L$ is stochastically continuous, i.e., for all $\eps>0$ and $t\in\bT$ holds
		\begin{equation*}
			\lim\limits_{\substack{s\to t,\\ s\in\bT}} \bP(\|L(t)-L(s)\|_U>\eps)=0.
		\end{equation*}
	\end{itemize}
	$L$ is called \textit{square-integrable} if\, $\bE(\|L(t)\|_U^2)<+\infty$ holds for any $t\in\bT$.
\end{defi}
We consider the SPDE
\be\label{eq:spde}
dX(t)=(AX(t)+F(t,X(t)))dt+G(t,X(t))dL(t), \quad X(0)=X_0,
\ee
on $\bT$, where $X_0$ is a $H$-valued random variable and $A: D(A)\subset H\to H$ is an unbounded, densely defined linear operator generating a $C_0$-semigroup $S=(S(t),t\ge0)\subset \cL(H)$ on $H$.
The coefficients $F$ and $G$ in Eq.~\eqref{eq:spde} are possibly non-linear measurable mappings  $F:\bT\times H\to H$ and $G:\bT\times H\to \LHS(\cU,H)$, respectively.
The driving noise is modeled by a square-integrable, $U$-valued L\'evy process $L$ defined on $(\gO,\cF,\bP, (\cF_t,t\ge0))$ with non-negative, symmetric and trace class covariance operator $Q\in \cL(U)$, satisfying the identity
\bee
\bE( (L(t)-\bE(L(t)),\phi)_U (L(t)-\bE(L(t)),\psi)_U )=t(Q\phi,\psi)_U,\quad \phi,\psi\in U,\;t\in\bT.
\eee
By the Hilbert-Schmidt theorem, the ordered eigenvalues $\eta_1\ge\eta_2\ge\dots\ge0$ of $Q$ are non-negative and have zero as their only accumulation point.
The corresponding eigenfunctions $(e_k,k\in\bN)\subset U$ form an orthonormal basis of $U$ and we define the \textit{square-root} of $Q$ via
\bee
Q^{1/2}\phi:=\sum_{k\in\bN}\sqrt{\eta_k}(\phi,e_k)_Ue_k,\quad\phi\in U.
\eee
Since $Q^{1/2}$ is not necessarily injective, the \textit{pseudo-inverse} of $Q^{1/2}$ is given by
\bee
Q^{-1/2}\varphi:=\phi,\quad\text{if $Q^{1/2}\phi=\varphi$\; and $\|\phi\|_U=\inf_{\varphi\in U,\,Q^{1/2}\varphi=\phi}\{\|\varphi\|_U\}$.}
\eee
With this, we are able to define the \textit{reproducing kernel Hilbert space} associated to $L$.

\begin{defi}
	Let $L$ be a square-integrable, $U$-valued L\'evy process with non-negative, symmetric, trace class covariance operator $Q\in \cL(U)$. Then, the set $\cU:=Q^{1/2}(U)$ equipped with the scalar-product
	\bee
	(\varphi_1,\varphi_2)_\cU:=(Q^{-1/2}\varphi_1,Q^{-1/2}\varphi_2)_U,\quad \varphi_1,\varphi_2\in\cU,
	\eee
	is called the \textit{reproducing kernel Hilbert space} (RKHS) of $L$.
\end{defi}

Note that $(\sqrt{\eta_k} e_k,k\in\bN)$ forms an orthonormal system in the RKHS $\cU$ and hence the norm on the space of Hilbert-Schmidt operators $\LHS(\cU,H)$ is given by
\bee
\|O\|^2_{\LHS(\cU,H)}=\sum_{k\in\bN}\eta_k \|Oe_k\|_H^2,\quad O\in \LHS(\cU,H).
\eee

\begin{ex}\label{ex:matern}
	An important special case is $U=L^2(\cD)$, where $\cD\subset\bR^d$ is an open and bounded spatial domain for $d\in\bN$ and $Q$ is the \textit{Mat\'ern covariance operator} with parameters $\nu,\rho>0$, given by
	\be\label{eq:matern}
	[Q\phi](x):=\int_\cD \frac{2^{1-\nu}}{\Gamma(\nu)}\left(\frac{\sqrt{2\nu}\|x-y\|}{\rho}\right)^\nu K_\nu\left(\frac{\sqrt{2\nu}\|x-y\|}{\rho}\right)\phi(y)dy,\quad \phi\in U,\; x\in\cD.
	\ee
	Above, $\gG$ is the Gamma function, $K_\nu$ is the modified Bessel function of the second kind with $\nu$ degrees of freedom, and $\left\|\cdot\right\|$ is an arbitrary norm on $\bR^d$, for instance the Euclidean norm. We refer to $\rho>0$ as the \textit{correlation length} of $Q$, while $\nu>0$ controls the spatial regularity of the paths generated by $Q$. More precisely, it holds that $L(t)(\cdot)\in C^{\ceil\nu -1}(\ol \cD)$ almost surely for each $t\in\bT$.
\end{ex}

Solutions of Problem~\eqref{eq:spde} are characterized in~\cite[Chapter 9]{PZ07}:
\begin{defi} \label{def:solutions}
	The \textit{predictable $\gs$-algebra} $\cP_\bT$ is the smallest $\gs$-field on $\gO\times\bT$ containing all sets of the form $\cA\times (s,t]$, where $\cA\in\cF_s$ and $s,t\in\bT$ with $s<t$.
	A $H$-valued stochastic process $Y:\gO\times\bT\to H$ is called \textit{predictable} if it is a $\cP_\bT$-$\cB(H)$-measurable mapping.
	The set of all square-integrable, $H$-valued predictable processes is denoted by
	\bee
	\cX_{\bT}:=\left\{Y:\gO\times\bT\to H|\,\text{Y is predictable and $\sup_{t\in\bT} \bE(\|Y(t)\|_H^2)<+\infty$} \right\}.
	\eee
	
	A process $X\in\cX_\bT$ is called 
	\begin{itemize}
		\item \textit{mild solution} to Eq.~\eqref{eq:spde} if
		\be \label{eq:mild}
		X(t)=S(t)X_0+\int_0^tS(t-s)F(s,X(s))ds+\int_0^tS(t-s)G(s,X(s))dL(s)
		\ee
		holds almost surely for all $t\in\bT$. 
		
		\item \textit{weak solution} to Eq.~\eqref{eq:spde} if
		\bee
		(X(t),v)_H=(X_0,v)_H+\int_0^t(X(s),A^*v)_H+(F(s,X(s)),v)_Hds+\int_0^t(G(s,X(s))^*v,dL(s))_\cU
		\eee
		holds almost surely for all $v\in D(A^*)$ and $t\in\bT$, where $A^*:D(A^*)\to H,\, G(s,v)^*\in L(H,\cU)$ are the adjoint operators to $A:D(A)\to H$ and $G(s,v)\in \LHS(\cU,H)$, respectively.
	\end{itemize}
\end{defi}

In Eq.~\eqref{eq:mild} $S:\bT\to \cL(H)$ is the semigroup generated by $A$, thus $S(t)=e^{tA}$ and Eq.~\eqref{eq:mild} may be interpreted as a \textit{variation-of-constants} formula.
In the definition of weak solutions, we use the identification $\LHS(\cU,\bR)=\cU$.
Hence, the integrand $s\mapsto G^*(s,X(s))^*v$ is a $\LHS(\cU,\bR)$-valued process and we obtain
\bee
\int_0^t(G(s,X(s))^*v,dL(s))_\cU:=\int_0^t G(s,X(s))^*vdL(s)=\left(v,\int_0^t G(s,X(s))dL(s)\right)_H
\eee
for any $v\in D(A^*)$, see \cite[Chapter 9.3]{PZ07}.
Solutions to Problem~\eqref{eq:spde} are infinite-dimensional processes, i.e. $X:\gO\times\bT\times\cD\to\bR$, where $\cD\subset\bR^d$ for some $d\in\bN$.
Therefore, in general $H\subseteq U=L^2(\cD)$.
To ensure that mild resp. weak solutions to \eqref{eq:spde} are well-defined and unique, we fix the following set of assumptions.
\begin{assumption}
	~
	\begin{enumerate}[label=(\roman*)] \label{ass1}
		\item $L$ is a centered, square integrable, $U$-valued L\'evy process with trace class covariance operator $Q$.
		\item $X_0\in L^2(\gO;H)$ is a $\cF_0$-measurable random variable.
		\item $A: D(A)\subset H\to H$ is densely defined and generates a $C_0$-semigroup $S=(S(t), t\ge 0)\subset\cL(H)$.
		\item \label{item:Lipschitz1} The mappings $F(\cdot,v):\bT\to H$ and $G(\cdot,v):\bT\to \LHS(\cU,H)$ are measurable for each $v\in H$ and there is a constant $C>0$ such that for all $t\in\bT$ and $v,w\in H$
		\bee
		\|F(t,v)-F(t,w)\|_H+\|G(t,v)-G(t,w)\|_{\LHS(\cU,H)} \le C\|v-w\|_H.
		\eee
	\end{enumerate}
\end{assumption}

\begin{rem}\label{rem:ass}
	~
	\begin{itemize}
		\item We focus on mean-square type convergence results in this article and only consider square-integrable processes $L$.
		This enables us to use the It\^o isometry in Lemma~\ref{lem:isometry} for stochastic integrals with respect to Hilbert space-valued L\'evy processes.
		Details on non-square integrable martingales as integrator can be found in \cite[Section 8.8]{PZ07}.
		\item If $L$ is of non-zero mean, then $E(L(t))=t\phi$ for some mean function $\phi\in U$. Hence, we may assume without loss of generality that $\bE(L(t))=0$ and incorporate $\phi$ as part of the nonlinearity $F$, if desired.
		\item Under Assumption~\ref{ass1}, all integrals in Definition~\ref{def:solutions} are well-defined, see \cite[Remark 9.6]{PZ07}.
		
		\item The global Lipschitz condition~\ref{item:Lipschitz1} with respect to the second argument is sufficient for existence and uniqueness of mild solutions.
		In the literature (e.g.\cite{LR15, PZ07}), slightly weaker assumptions of the form
		\bee
		\|S(t)(F(s,v)-F(s,w))\|_{H}\le b_F(t,s)\|v-w\|_{H},\quad s,t\in(0,T]
		\eee
		for $b_F\in L^2(\bT\times\bT)$ are imposed.
		For the numerical analysis in the forthcoming chapters, however, we utilize the weak solution of the SPDE, and it is therefore advantageous to assume Lipschitz continuity of $F$ and $G$.
		We note that this condition on $F$ and $G$ implies the linear growth bound
		\begin{align*}
			\|F(t,v)\|_H+\|G(t,v)\|_{\LHS(\cU,H)}&\le C(1+\|v\|_H),\quad v\in H,\; t\in\bT.
		\end{align*}
	\end{itemize}
\end{rem}

\begin{thm}\label{thm:exis}
	Under Assumption~\ref{ass1}, there exist a unique mild and weak solution $X\in\cX_{\bT}$ to Problem~\eqref{eq:spde}.
	Both solutions coincide and there exists $C=C(\bT)>0$, independent of $X_0$, such that
	\bee
	\|X(t)\|_{L^2(\gO;H)}\le C(1+\|X_0\|_{L^2(\gO;H)}),\quad t\in\bT.
	\eee
\end{thm}
\begin{proof}
	Existence and uniqueness of a mild solution as in Eq.~\eqref{eq:mild} is proven in detail in \cite[Theorem 9.29]{PZ07}. We only sketch the main idea here for later reference.
	Let $\vartheta>0$ be arbitrary and define the norm
	\begin{equation*}
		\|Y\|_\vartheta:=\sup_{t\in\bT}e^{-\vartheta t}\bE(\|Y(t)\|_H^2)^{1/2},
		\quad
		Y\in \cX_\bT.
	\end{equation*}
	With this, $(\cX_\bT,\|\cdot\|_\vartheta)$ is a Banach space, and, using $X_0$ as initial value, a sequence of fixed-point iterations on $(\cX_\bT,\|\cdot\|_\vartheta)$ is given by
	\bee
	X_{n+1}(t)=\Psi(X_n)(t):= S(t)X_0 + \int_0^t S(t-s)F(s,X_n(s))ds + \int_0^t S(t-s)G(s,X_n(s))dL(s),\quad n\in\bN_0.
	\eee
	Under Assumption~\ref{ass1}, and by choosing $\gb>0$ large enough, $\Psi:\cX_\bT\to \cX_\bT$ is a contraction, so that existence and uniqueness of mild solutions follow by Banach's fixed-point theorem.
	The equivalence of weak and mild solutions is shown in \cite[Theorem 9.15]{PZ07}.
\end{proof}

We further record the following bound on $C_0$-semigroups for later reference. 

\begin{lem}\cite[Chapter 1.2]{P83}\label{lem:semigroup}
	Let $S=(S(t),t\ge0)$ be a $C_0$-semigroup with infinitesimal generator $A$ on a Banach space $(\cY,\|\cdot\|_\cY)$.
	Then, there are constants $C_1,C_2>0$ such that for all $\phi\in \cY$ and $t\ge0$
	\bee
	\|S(t)\phi\|_\cY\le C_1e^{C_2 t}\|\phi\|_\cY.
	\eee
\end{lem}

In the remainder of this article, we investigate the case where $A$ is a first order differential operator and Eq.~\eqref{eq:spde} is a (hyperbolic) transport equation with L\'evy noise.
The next section establishes temporal and spatial regularity of $X$ in this scenario to pave the way for a numerical analysis in Sections~\ref{sec:dg}-\ref{sec:noise}.

To conclude this section, we record infinite-dimensional versions of It\^o's formula and the It\^o isometry, that require some more notation as preparation.
Let $(f_i, i\in\bN)$ be an orthonormal basis of the (separable) Hilbert space $H$ and denote by $H\otimes H$ the corresponding tensor-product Hilbert space. 
Let $H\widehat\otimes H$ be the completion of $H\otimes H$ with respect to the Hilbert-Schmidt norm 
\begin{equation*}
	\|\phi\otimes\varphi\|_{HS}^2 := \sum_{i,j\in\bN} \|(\phi, f_i) f_i\|_H^2 \|(\varphi, f_j) f_j\|_H^2, 
	\quad \phi\otimes\varphi\in H\otimes H. 
\end{equation*}
For a given $H$-valued semimartingale $Y:\gO\times \bT\to H$, we consider a fixed decomposition $Y=M+N$, where $M:\gO\times \bT\to H$ is the local martingale part of $Y$ and the process $N:\gO\times \bT\to H$ has paths of bounded variation.  
For any real-valued semimartingale $Y^{(1)}:\gO\times\bT\to\bR$ we denote the quadratic variation process of $Y^{(1)}$ by $([Y^{(1)}, Y^{(1)}]_t, t\in\bT)$. The quadratic covariation process of the two real-valued semimartingales $Y^{(1)}, Y^{(2)}:\gO\times\bT\to\bR$ is defined by the polarization identity 
\begin{equation*}
	[Y^{(1)}, Y^{(2)}]_t:=\frac{1}{2}\left([Y^{(1)} + Y^{(2)}]_t-[Y^{(1)}]_t - [Y^{(2)}]_t\right),
	\quad t\in\bT.
\end{equation*}
With this at hand, the quadratic variation proces of a given $H$-valued (local) martingale $M$ is given by the $H\widehat \otimes H$-valued process
\begin{equation*}
	[M, M]_t:=\sum_{i,j\in\bN} f_i\otimes f_j\, [(M, f_i)_H, (M, f_j)_H]_t, \quad t\in\bT.
\end{equation*}

\begin{lem}
	(It\^o formula, \cite[Theorem D.2]{PZ07})\label{lem:ito-formula}
	Let $Y=M+N$ be a $H$-valued semimartingale and let $\Psi:H\to \bR$ be such that  there holds $D\Psi(\phi)\in H$ for all $\phi\in \cL(H, \bR)\simeq H$ and that the mapping 
	$$H\ni \phi\mapsto D^2\Psi(\phi)\in \cL_{HS}(H)\simeq \cL(H\widehat \otimes H, \bR)$$
	is uniformly continuous on any bounded subset of $H$. 
	
	Then, $\Psi(Y)$ is a local semimartingale and there holds $P$-a.s. for all $t\in\bT$ 
	\begin{equation}\label{eq:ito_formula}
		\begin{split}
			\Psi(Y(t)) &= \Psi(Y(0))
			+ \int_0^t \left(D\Psi(Y(s-)), dY(s) \right)_H
			+\frac{1}{2} \int_0^t D^2\Psi(Y(s-)) d[M,M]_s \\
			&\quad +\sum_{s\le t} \gD (\Psi(Y))(s) - \left( D\Psi(Y(s-)), \gD Y(s) \right)_H 
			 -\frac{1}{2}\sum_{s\le t} D^2\Psi(Y(s-)) (\gD Y(s)\otimes \gD Y(s)).
		\end{split}
	\end{equation}
	
\end{lem}

\begin{lem}(It\^o isometry, \cite[Corollary 8.17]{PZ07})\label{lem:isometry}
	Let $\gk:\gO\times\bT\to \LHS(\cU, H)$ be a predictable, square integrable process and let $L$ satisfy Assumption~\ref{ass1}(i).
	Then, $\gk$ is an admissible integrand for $L$, and for all $t\in\bT$ it holds that
	\bee
	\bE\left(\left\|\int_0^t  \gk(s)dL(s) \right\|_{H}^2\right)
	=\int_0^t\bE\left(\|\gk(s)\|^2_{\LHS(\cU, H)}\right)ds
	=\int_0^t\bE\left(\sum_{k\in\bN}\eta_k\|\gk(s)e_k\|^2_{H}\right)ds.
	\eee
\end{lem}

\section{The Stochastic Transport Equation} \label{sec:ste}

Let us regard Eq.~\eqref{eq:spde} with respect to a convex spatial domain $\cD\subset\bR^d$ with $d\in \{1,2\}$, i.e. the solution $X$ is a $H$-valued process with $H=L^2(\cD)$.
We denote for $k\in\bN$ the standard Sobolev space $H^k:=H^k(\cD)$ equipped with the norm, resp. seminorm
\begin{equation*}
	\|v\|_{H^k}:=\left(\sum_{|\ga|\le k}\int_\cD |D^\ga v(x)|^2dx\right)^{1/2},\quad |v|_{H^k}:=\left(\sum_{|\ga|= k}\int_\cD |D^\ga v(x)|^2dx\right)^{1/2},
\end{equation*}
where $D^\ga=\partial_{x_1}^{\ga_1}\dots\partial_{x_d}^{\ga_d}$ is the mixed partial weak derivative (in space) with respect to the multi-index $\ga\in\bN_0^d$ and $|\ga|:=\sum_{i=1}^d\ga_i$.
The fractional order Sobolev spaces $H^\gg$ are defined by
\begin{align*}
H^\gg:&=\{v\in H|\; \|v\|_{H^\gg}<\infty\},\quad \gg>0, \\
\|v\|^2_{H^\gg}&=\|v\|^2_{H^{\floor \gg}}+\sup_{|\ga|= \floor \gg}|D^\ga v|^2_{H^{(\gg-\floor \gg)}}
:=\|v\|^2_{H^{\floor \gg}}+\sup_{|\ga|= \floor \gg}\int_\cD\int_\cD\frac{|D^\ga v(x)-D^\ga v(y)|^2}{|x-y|^{d+2(\gg-\floor \gg)}}dxdy,
\end{align*}
where the last term
is the \emph{Gagliardo seminorm}, see, e.g., \cite{DGV12}.
Let $a\in\bR^d\setminus\{0\}$ be a fixed vector, and let $A=a\cdot\nabla$ in Eq.~\eqref{eq:spde} be the first order differential operator. This yields the \textit{stochastic transport problem}
\be\label{eq:ste}
dX(t)=(a\cdot\nabla X(t)+F(t,X(t)))dt+G(t,X(t))dL(t), \quad X(0)=X_0.
\ee
The inflow boundary of $\cD$ is given by
\bee
\partial \cD^+:=\{x\in\partial\cD: a\cdot\vv n(x)>0\},
\eee
where $\vv n$ is the exterior normal vector to $\partial\cD$. 
We further define 
\bee
\partial \cD^-:=\{x\in\partial\cD: a\cdot\vv n(x)<0\}
\quad\text{and}\quad
\partial \cD^0:=\{x\in\partial\cD: a\cdot\vv n(x)=0\},
\eee 
as the outflow and characteristic boundary, respectively.
We equip Eq.~\eqref{eq:ste} with homogeneous inflow boundary conditions $X(t)=0$ on $\partial\cD^+$ for all $t\in\bT$.

\begin{rem}\label{rem:BC}
	Homogeneous inflow boundary conditions are imposed for notational convenience, and are not restrictive in our setting.
	In our example in Section~\ref{sec:numerics}, we examine an energy forward model with nonzero, but constant inflow boundary condition $X(t)=c>0$ on $\partial\cD^+$ for all $t\in\bT$.
	To treat this case in our setting, let $A, X_0, F, G$ and $L$ be given, and let $X:\gO\times\bT\to H$ be a solution to Eq.~\eqref{eq:ste_weak}, but with inhomogeneous boundary conditions $X(t)=c\in\bR$ on $\partial\cD^+$.
	For any $c\in\bR$ we define $X^{hom}(t):=X(t)-c$, as well as the modified coefficients
	\bee
	F^{hom}(s,v):=F(s,v+c)\quad\text{and}\quad
	G^{hom}(s,v):=G(s,v+c).
	\eee
	Note that if $F$ and $G$ satisfy Assumption~\ref{ass1}\ref{item:Lipschitz1}, then $F^{hom}$ and $G^{hom}$ also satisfy Assumption~\ref{ass1}\ref{item:Lipschitz1}.
	It is then readily verified that $X^{hom}(t)=0$ on $\bT\times\partial\cD^+$ and for all $t\in[0,T]$ it holds
	\begin{align*}
		dX^{hom}(t)=(a\cdot\nabla X^{hom}(t)+F^{hom}(t,X^{hom}(t))dt+G^{hom}(t,X^{hom}(t))dL(t), \quad X^{hom}(0)=X_0-c.
	\end{align*}
\end{rem}

To derive a weak formulation of Eq.~\eqref{eq:ste} in $H=L^2(D)$, we follow the approach for deterministic transport problems from \cite[Section 2.2]{DHSW12}:
For any $v,w\in C(\ol\cD)\cap C^1(\cD)$ Green's identity yields
\bee
(Aw,v)_H=(w,A^*v)_H+\int_{\partial\cD^+}a\cdot\vv n wvdz+\int_{\partial\cD^-}a\cdot\vv n wvdz,
\eee
where $A^*= - a \cdot \nabla $ is the formal adjoint of $A$.
Now let
\bee
C^1_{\triangle}(\cD):=\{v\in C(\ol\cD)\cap C^1(\cD) \big|\;v|_{\partial\cD^\triangle}=0 \},
\quad \triangle\in\{+,-\}.
\eee
For any $v\in C^1_-(\cD)$ we define $\|v\|_V:=\|A^*v\|_H$, and note that $\|\cdot\|_V$ is a norm on $C^1_-(\cD)$, since $A^*$ is injective on $C^1_-(\cD)$ for $a\neq 0$ (\cite[Remark 2.2, case (i)]{DHSW12}).
Furthermore, let
\be\label{eq:V}
	V:=\textrm{clos}_{\|\cdot\|_V} C^1_-(\cD)\subset H,
\ee
where $\textrm{clos}_{\|\cdot\|}O$ signifies the closure of a set $O$ with respect to a given norm $\left\|\cdot\right\|$. 
We extend $A$ from $D(A)=\{v\in H|\; Av\in H\}\subset H$ to $H$ in a distributional sense, so that $A:D(A)\to H$ and $A:H\to V'$.

\begin{lem}\label{lem:operator-inverse}
	Let $V'$ be the topological dual of $V$ as in~\eqref{eq:V} and let $a\in\bR^d\setminus\{0\}$. 
	It holds that $V=D(A^*)$, the embeddings $V\hookrightarrow H\hookrightarrow  V'$ are dense, and that 
	$$A^*:D(A^*)=V\to H,\quad\text{and}\quad
	A:H\to V'$$ 
	are isometric, bounded linear operators.
\end{lem}
\begin{proof}
	Let $v\in D(A^*)$ and $w:=A^*v \in H$. 
	As $a\neq0$, we obtain from \cite[Remark 2.2, case (i)]{DHSW12} that $A^*[C_-^1(\cD)]$ is dense in $H$.
	Thus, there is a sequence $(v_n, n\in\bN)\subset C_-^1(\cD)$ such that 
	$A^*v_n\to w = A^*v$ in $H$ as $n\to\infty$. 
	Hence, $v\in V= \textrm{clos}_{\|\cdot\|_V} C^1_-(\cD)$.  
	The remaining parts of the claim are given by \cite[Remark 2.2 and Proposition 2.1]{DHSW12}.
\end{proof}

We denote the inverse operators of $A^*$ and $A$ by 
$$A^{-*}:= (A^{*})^{-1}: H\to V,\quad\text{and}\quad A^{-1}:V'\to H,$$
and note that Lemma~\ref{lem:operator-inverse} implies for any $w\in H$ that
\begin{equation}\label{eq:dual-estimate}
	\|A^{-1}w\|_{H}=\|w\|_{V'}\le \sup_{v\in V} \frac{(w, v)_H}{\|v\|_V} 
	\le \|w\|_H \sup_{v\in V} \frac{\|A^{-*}\|_{\cL(H, V)}\|A^*v\|_H}{\|v\|_V} = \|w\|_H. 
\end{equation}
We further define the bilinear form
\be \label{eq:blf}
B:H\times V\to\bR,\quad (w,v)\mapsto(w,-A^*v)_H=\dualpair{V'}{V}{-Aw}{v}.
\ee
Clearly, $B(w,v)\le \|w\|_H \|v\|_V$ for $(w,v)\in H\times V$, and, since $A^*:V\to H$  is an isometry, it follows that $B$ satisfies the \emph{inf-sup condition}
\be\label{eq:infsup}
\inf_{w\in H, w\neq0} \sup_{v\in V, v\neq0} \frac{B(w,v)}{\|w\|_H \|v\|_V}\ge 1,
\ee
see \cite[Section 2.2]{DHSW12} for a detailed derivation.
Moreover, partial integration shows that 
\be\label{eq:skew-symmetry}
B(w,v)+B(v,w)
= (w,a\cdot\nabla v)_H + (v,a\cdot\nabla w)_H
=\int_{\partial\cD^+} a\cdot n(z) w(z)v(z)dz  
\quad v, w\in V.
\ee
As $v|_{\partial\cD^-}=0$ and $a\cdot n(z)>0$ on $\partial\cD^+$ we may define the seminorm
\be\label{eq:inflownorm}
|v|_{+,a}^2:=2B(v,v)=-2(v,A^*v)_H=\int_{\partial\cD^+} a\cdot n(z) v(z)^2dz\ge 0,
\quad v\in V.
\ee

The weak formulation of Eq.~\eqref{eq:ste} is now to find $X:\gO\times\bT\to H$ such that for all $v\in D(A^*)$ it holds
\be \label{eq:ste_weak}
\begin{split}
	(X(t),v)_H+\int_0^tB(X(s),v)ds&=(X_0,v)_H+\int_0^t(F(s,X(s)),v)_Hds
	+\left(\int_0^tG(s,X(s))dL(s), v\right)_H.
\end{split}
\ee

The numerical schemes to approximate $X$ and the corresponding error estimates are based on the weak formulation from Eq.~\eqref{eq:ste_weak}. As we see in Theorem~\ref{thm:spatial_reg}, however, mild solutions to Eq.~\eqref{eq:ste} are convenient to investigate the spatial regularity of $X$.
To this end, we show that the operator $A=a\cdot\nabla$ is the infinitesimal generator of a semigroup $S$ on $H$, namely the \textit{shift semigroup} given by
\be\label{eq:shift_semigroup}
[S(t)w](x):=\begin{cases}
	w(at+x)\quad&\text{if $at+x\in\cD$}\\
	0\quad&\text{if $at+x\notin\cD$}
\end{cases}.
\ee

\begin{lem}
	The family of operators $(S(t),t\ge0)$ defined in Eq.~\eqref{eq:shift_semigroup} forms a $C_0$-semigroup of bounded linear operators on $H$. Furthermore, the infinitesimal generator of $S$ is given by $A:D(A)\to H$.
\end{lem}
\begin{proof}
	By the definition of $S$, it is immediate that $\|S(t)w\|_H\le\|w\|_H$, $S(0)=I$, and $S(t+s)=S(t)S(s)$ for $t,s\in\bT$.
	Hence, $(S(t),t\ge0)$ is a semigroup of bounded linear operators on $H$.
	To see that $(S(t),t\ge0)$ is strongly continuous, let $w\in C_c^0(\cD)\subset H$ be a compactly supported, continuous function on $\cD$.
	Furthermore, let $\widetilde w\in C_c^0(\bR^d)$ be the zero-extension of $w$ on $\bR^d$ given by
	\be\label{eq:zeroext}
	\widetilde w(x):=\begin{cases}
		w(x)\quad&\text{if $x\in\cD$}\\
		0\quad&\text{if $x\in\bR^d\setminus\cD$}
	\end{cases}.
	\ee
	This yields
	\bee
	\lim_{t\to 0}\|S(t)w-w\|^2_H=\lim_{t\to 0}\int_{\cD} (\widetilde w(at+x)))-\widetilde w(x))^2dx
	=\int_{\cD} \lim_{t\to 0}(\widetilde w(at+x)))-\widetilde w(x))^2dx=0.
	\eee
	Note that the interchange of limit and integral is justified, since $\widetilde w$ is uniformly bounded on $\bR^d$.
	The last identity holds due to continuity of $\widetilde w$ on $\bR^d$, which is in turn given since $w$ is compactly supported in the open set $\cD$.
	By density of $C^0_c(\cD)$ in $H$, it follows that $S$ is a $C_0$-semigroup on $H$.
	
	For the second part of the claim, we need to verify that for all $w\in D(A)\subset H$ it holds that
	\be\label{eq:generator}
	\lim_{t\to 0}\left\|\frac{S(t)w-w}{t}-Aw\right\|_H=0.
	\ee	
	To this end, we observe that 
	\begin{align*}
		D(A)&=\{w\in H|\,
		\text{there is $c=c(w)>0$ such that $|(w, A^*v)_H|\le c \|v\|_H$ for all $v\in D(A^*)$}
		\}\\
		&=
		\textrm{clos}_{\|A\cdot\|_H} C^1_+(\cD).
	\end{align*}
	It is therefore sufficient to verify~\eqref{eq:generator} for $w\in C^1_+(\cD)$. 
	Observe that
	for any fixed $x\in\cD$, there is a $t_x>0$ such that $at+x\in\cD$ for all $t\in(0,t_x)$.
	Hence, multidimensional Taylor expansion yields for some $\gt\in(0,1)$
	\bee
	\lim_{t\to 0}\frac{[S(t)w](x)-w(x)}{t}
	=\lim_{t\to 0}\left(\frac{w(at+x)-w(x)}{t} \right)
	= \lim_{t\to 0} a\cdot\nabla w(\gt at + x)
	= a\cdot\nabla w(x).
	\eee
	The last limit holds as $\nabla w$ is continuous on $\cD$. 
	Thus, Equation~\eqref{eq:generator} holds for
	for all $w\in C_+^1(\cD)$, and the claim follows since $C_+^1(\cD)$ is dense in $D(A)$.
	 \end{proof}

Regarding the mild solution in~\eqref{eq:mild},
we see that the shift by $S$ as in~\eqref{eq:shift_semigroup} may introduce (spatial) discontinuities or "kinks" if $X_0, F$ and $G$ do not decay smoothly to zero at the inflow boundary.
To account for this and derive the spatial regularity of $X$ we have to strengthen Assumption~\ref{ass1}:

\begin{assumption}\label{ass2}
	For $A,F,G,L$ and $X_0$ given as in~\eqref{eq:ste}, and the orthonormal eigenpairs $((\eta_k, e_k),k\in\bN)\subset \bR_0\times U$ of $Q$ let the following hold:
	\begin{enumerate}[label=(\roman*)]
		\item \label{item:EV} $L$ is a square integrable, $U$-valued L\'evy process with zero mean and trace class covariance operator $Q$. The eigenvalues $(\eta_k,k\in\bN)$ of $Q$ are in decreasing order and there are constants $\ga>1$ and $C>0$ such that $\eta_k\le Ck^{-\ga}$ for all $k\in\bN$.
		
		\item \label{item:X0} There are constants $\gg_0>0 $ and $C>0$ such that $X_0\in L^2(\gO;H^{\gg_0})$ is a $\cF_0$-measurable random variable, and for all $t\ge0$ it holds that $\bE(\|S(t)X_0\|_{H^{\gg_0}}^2)\le C  \bE(\|X_0\|_{H^{\gg_0}}^2)$.
		
		\item \label{item:Lipschitz2}  $F:\bT\times H\to H$ and $G:\bT\times H\to \LHS(\cU,H)$ are H\"older continuous with exponent $\frac{1}{2}$ on $\bT$ and
		globally Lipschitz on $H$, i.e. for all $ v,w\in H$ and $s,t\in\bT$ it holds that
		\begin{align*}
			\|F(t,v)-F(s,v)\|_H+\|G(t,v)-G(s,v)\|_{\LHS(\cU,H)}\le C|t-s|^{\frac{1}{2}}(1+\|v\|_H),
		\end{align*}
		and
		\bee
		\|F(t,v)-F(t,w)\|_H+\|G(t,v)-G(t,w)\|_{\LHS(\cU,H)}\le C\|v-w\|_H.
		\eee
			
		\item \label{item:lingrowthF}
		There are constants $\gg_F>0 $ and $C>0$ such that for all $v\in H^{\gg_F}$ and $s,t\in\bT$ with $s\le t$ it holds
		\bee
		\|S(t-s)F(t,v)\|_{H^{\gg_F}}\le C(1+\|v\|_{H^{\gg_F}}).
		\eee
		
		\item \label{item:lingrowthG}
		There are constants $\gg_G>0 $, $\gb\in(0,(\ga-1)/2\ga)$ (where $\ga>1$ is as in~\ref{item:EV}) and $C>0$ such that for all $v\in H^{\gg_G}$, $k\in\bN$ and $s,t\in\bT$ with $s\le t$ it holds that
		\bee
		\|S(t-s)G(t,v)e_k\|_{H^{\gg_G}}\le C(1+\|v\|_{H^{\gg_G}})\eta_k^{-\gb}.
		\eee
	\end{enumerate}
\end{assumption}

\begin{rem}\label{rem:ass2}
	~
	\begin{itemize}
		\item For functions $v\in H^\gg$ with arbitrary large $\gg>0$, the shift by $S(t)$ results in a discontinuous function and hence $S(t)v\in H^{1/2-\eps}(\cD)$ for any $\eps>0$ and $t>0$. In our experiments in Section~\ref{sec:numerics}, we consider functions that vanish on $\partial\cD^+$, but have nonzero derivatives at $\partial\cD^+$, for which assumption \ref{ass2}\ref{item:lingrowthF} and \ref{item:lingrowthG} holds with improved regularity $\gg:=min(\gg_0,\gg_F,\gg_G)\in (1/2, 3/2)$.
		
		\item 	Assumption~\ref{ass2}\ref{item:Lipschitz2} on the H\"older continuity with respect to $\bT$ is necessary to ensure the rate of convergence of the time stepping scheme introduced in Section~\ref{sec:time}. This condition implies in particular that $F$ and $G$ are measurable with respect to $\bT$.
	\end{itemize}
\end{rem}

\begin{ex}\label{ex:matern2}
	Recall Example~\ref{ex:matern} with $U=L^2(\cD)$, $Q$ as the Mat\'ern covariance operator from Eq.~\eqref{eq:matern}
	with smoothness parameter $\nu>0$ and assume for any $\gg<\nu$ that
	\be\label{eq:g_est}
	\|S(t-s)G(s,v)e_k\|_{H^\gg}\le C(1+\|v\|_{H^\gg})\|e_k\|_{H^\gg}.
	\ee
	By \cite[Proposition 9]{Gr15}, Assumption~\ref{ass2}\ref{item:EV} holds with $\ga=1+2\nu/d>1$.
	Moreover, if $\nu>d/2$, the proof of \cite[Proposition 9]{Gr15} yields for all $q,\widetilde q$ such that $0\le q\le \widetilde q < d+2\nu$ the estimate
	\bee
	\|e_k\|_{H^q}\le C\eta_k^{-q/\widetilde q}.
	\eee
	Now let $q=\nu-\eps_1$ and $\widetilde q=d+2\nu-\eps_2$, where $\eps_1>0$ is arbitrary small and $\eps_2\in(0,\eps_1(d/\nu+2))$.
	By construction, $\gb=q/\widetilde q$ satisfies $0<\gb<\nu/(d+2\nu)=(\ga-1)/2\ga$ and Ineq.~\eqref{eq:g_est} yields for $\gg_G:=q$ that
	\bee
	\|G(s,v)e_k\|_{H^{\gg_G}}\le C(1+\|v\|_{H^{\gg_G}})\eta_k^{-\gb}.
	\eee
	Regarding the eigenpairs of $Q$, Assumption~\ref{ass2}\ref{item:lingrowthG} is therefore satisfied for any $\gg_G<\nu$ and $\gb<\nu/(d+2\nu)$ in the Mat\'ern case, provided that $\nu>d/2$. Thus, we may infer the (maximum) "mean-square spatial regularity" of $X$ directly from $\nu$ (see Theorem~\ref{thm:spatial_reg}).
\end{ex}

\begin{thm}\label{thm:spatial_reg}
	Let Assumption~\ref{ass2} hold, define $\gg:=min(\gg_0,\gg_F,\gg_G)>0$ and $\iota:=min(\gg_0,\gg_F)>0$.
	There exists a unique solution $X$ to Eq.~\eqref{eq:ste_weak} and Eq.~\eqref{eq:mild}, and both solutions coincide almost surely. Moreover, there holds
	\bee
	\sup_{t\in\bT} \|X(t)\|^2_{L^2(\gO;H^\gg)}\le C\left(1+\|X_0\|^2_{L^2(\gO;H^\gg)}\right)<+\infty,
	\eee
	and
	\bee
	\sup_{t\in\bT} \|\bE(X(t))\|_{H^\iota}
	\le C\left(1+\|X_0\|_{L^1(\gO;H^\iota)}\right)<+\infty.
	\eee
\end{thm}
\begin{proof}
	Existence, uniqueness and equivalence of a weak and mild solution $X:\gO\times \bT\to H$ follow by Theorem~\ref{thm:exis}, since Assumption~\ref{ass2} implies in particular Assumption~\ref{ass1}.
	We recall that the mild solution $X\in\cX_\bT$ is the unique fixed-point of the iteration
	\bee
	X_n(t)=\Psi(X_{n-1})(t):=S(t)X_0+\int_0^tS(t-s)F(s,X_{n-1}(s))ds+\int_0^tS(t-s)G(s, X_{n-1}(s))dL(s),
	\quad n\in\bN,
	\eee
	and with $X_0(t):=X_0$ for all $t\in\bT$. To derive the spatial regularity of $X$, we use a similar strategy of proof as for the existence and uniqueness proof of mild solutions.
	We introduce the space
	\bee
	\cX_{\bT,\gg}:=\{Y:\gO\times\bT\to H^\gg|\,\text{$Y$ is predictable and $\sup_{t\in\bT} \bE(\|Y(t)\|_{H^\gg}^2)<+\infty$} \}\subset \cX_\bT,
	\eee
	with the weighted norm
	\bee
	\|Y\|_{\vartheta,\gg}:=\sup_{t\in\bT}e^{-\vartheta t}\bE(\|Y(t)\|_{H^\gg}^2)^{1/2},\quad Y\in \cX_{\bT,\gg},\quad \vartheta>0.
	\eee
	We first show by induction that $\|X_n\|_{\vartheta,\gg}<+\infty$ holds for all $n\in\bN$. Then,  we choose $\vartheta>0$ large enough so that $\Psi$ is a contraction on $(\cX_{\bT,\gg},\|\cdot\|_{\vartheta,\gg})$ to obtain a uniform (in $n\in\bN$) estimate for $\|X_n\|_{\vartheta,\gg}$.
	
	Starting with $X_1=\Psi(X_0)$, we note that $X_0\in \cX_{\bT,\gg}$ by Assumption~\ref{ass2}\ref{item:X0}, and obtain
	\bee
	\|S(\cdot)X_0\|^2_{\vartheta,\gg}
	=\sup_{t\in\bT}e^{-\vartheta t}\|S(t)X_0\|^2_{L^2(\gO;H^\gg)}
	\le \sup_{t\in\bT}e^{-\vartheta t} C\|X_0\|^2_{L^2(\gO;H^\gg)}
	=C \|X_0\|^2_{L^2(\gO;H^\gg)}<+\infty.
	\eee
	Jensen's inequality and Assumption~\ref{ass2}\ref{item:lingrowthF} yield similarly
	\begin{align*}
		\left\|\int_0^\cdot S(\cdot-s)F(s,X_0(s))ds\right\|^2_{\vartheta,\gg}	&=\sup_{t\in\bT}\int_0^te^{\vartheta(s-t-s)}\|S(t-s)F(s,X_0(s))\|^2_{L^2(\gO;H^\gg)}ds\\
		&\le C\sup_{t\in\bT}\int_0^t e^{\vartheta (s-t)}e^{-\vartheta s}\left(1+\|X_0(s)\|^2_{L^2(\gO;H^\gg)}\right)ds\\
		&\le C\left(T+\sup_{t\in\bT}\int_0^t e^{\vartheta (s-t)}ds \|X_0\|^2_{\vartheta,\gg}\right)\\
		&\le C\left(1+\frac{1-e^{-\vartheta T}}{\vartheta} \|X_0\|^2_{\vartheta,\gg}\right).
	\end{align*}
	Moreover, the It\^o isometry from Lemma~\ref{lem:isometry} shows for all $t\in\bT$ the identity
	\begin{align*}
		\left\|\int_0^tS(t-s)G(s,X_0(s))dL(s)\right\|_{L^2(\gO;H^\gg)}^2
		&=\bE\left(\int_0^t\|S(t-s)G(s,X_0(s))\|^2_{\LHS(\cU,H^\gg)}ds\right)\\
		&=\bE\left(\int_0^t\sum_{k\in\bN}\eta_k\|S(t-s)G(s,X_0(s))e_k\|^2_{H^\gg}ds\right).
	\end{align*}
	With Assumption~\ref{ass2}\ref{item:lingrowthF} this yields the estimate
	\begin{align*}
		\left\|\int_0^\cdot S(\cdot-s)G(s,X_0(s))dL(s)\right\|_{\vartheta,\gg}^2
		&\le C\sum_{k\in\bN}\eta_k^{1-2\gb}\sup_{t\in\bT}e^{-\vartheta t}\int_0^t \left(1+\|X_0(s)\|^2_{L^2(\gO;H^\gg)}\right)ds\\
		&\le C\sum_{k\in\bN}k^{-\ga(1-2\gb)}\left(T+\sup_{t\in\bT}\int_0^te^{\vartheta(s-t-s)}\|X_0(s)\|^2_{L^2(\gO;H^\gg)}ds\right)\\
		&\le C\sum_{k\in\bN}k^{-\ga(1-2\gb)}\left(1+\frac{1-e^{-\vartheta T}}{\vartheta} \|X_0\|^2_{\vartheta,\gg}\right).
	\end{align*}
	Since $\ga(1-2\gb)>1$ by Assumption~\ref{ass2}\ref{item:lingrowthG}, it holds that $\sum_{k\in\bN}k^{-\ga(1-2\gb)}<+\infty$, and we obtain
	\begin{align*}
		\|X_1\|^2_{\vartheta,\gg}
		&\le C\left(1+\|X_0\|^2_{L^2(\gO;H^\gg)}+\frac{1-e^{-\vartheta T}}{\vartheta}\|X_0\|^2_{\vartheta,\gg}\right)<+\infty,
	\end{align*}
	where $C>0$ is finite and uniformly bounded in $\bT$.
	We now iterate this estimate over $n\in\bN$ and obtain
	\begin{align*}
		\|X_n\|^2_{\vartheta,\gg}
		&\le C\left(1+\|X_0\|^2_{L^2(\gO;H^\gg)}+\frac{1-e^{-\vartheta T}}{\vartheta}\|X_{n-1}\|^2_{\vartheta,\gg}\right)\\
		&\le \left(1+\|X_0\|^2_{L^2(\gO;H^\gg)}\right)\sum_{k=0}^{n}C^k\left(\frac{1-e^{-\vartheta T}}{\vartheta}\right)^k.
	\end{align*}
	We choose $\vartheta>0$ so that $C(1-e^{-\vartheta T})/\vartheta<1$ (e.g., $\vartheta=C$), which gives a uniform bound with respect to $n$ for the last estimate.
	Taking the limit $n\to\infty$ and multiplying by $e^{\vartheta T}$ yields the first part of the claim.

	To prove the second part, observe that Assumption~\ref{ass1} implies for any $Y\in L^1(\gO; H^{\gg_F})$ and $s,t\in\bT$ with $s\le t$ that
	\begin{align*}
		\|\bE(S(t-s)F(s, Y))\|_{H^{\gg_F}}
		&\le 
		\|S(t-s)F(s, \bE(Y))\|_{H^{\gg_F}} 
		+ \|\bE(S(t-s)F(s, Y)-S(t-s)F(s, \bE(Y)))\|_{H^{\gg_F}} \\
		&\le 
		C(1+\|\bE(Y)\|_{H^{\gg_F}}
		+ \bE(\|Y-\bE(Y)\|_{H^{\gg_F}})) \\
		&\le 
		C(1+\|Y\|_{L^1(\gO;H^{\gg_F})}). 
	\end{align*}
	
	Using that the stochastic integral in the fixed-point iteration for $X_n$ vanishes in expectation, we find by similar calculations as above for 
	$\iota=\min(\gg_0,\gg_F)$ that 
	\begin{align*}
		\sup_{t\in\bT}e^{-\vartheta t}\|\bE(X_1(t))\|_{H^\iota}
		&\le C\left(
		1+\|X_0\|_{L^1(\gO;H^\iota)}
		+\frac{1-e^{-\vartheta T}}{\vartheta}\|\bE(X_0)\|_{H^\iota}
		\right).
	\end{align*}
	The proof is then finished by again iterating the esimates over $n$ as there holds
	\begin{align*}
		\sup_{t\in\bT}e^{-\vartheta t}\|\bE(X_n(t))\|_{H^\iota}
		&\le C\left(
		1+\|X_0\|_{L^1(\gO;H^\iota)}
		+\frac{1-e^{-\vartheta T}}{\vartheta}\sup_{t\in\bT}e^{-\vartheta t}\|\bE(X_{n-1}(t))\|_{H^\iota}
		\right) \\
		&\le C\left(
		1+\|X_0\|_{L^1(\gO;H^\iota)}\right)
		\sum_{k=0}^{n}C^k\left(\frac{1-e^{-\vartheta T}}{\vartheta}\right)^k.
	\end{align*}
\end{proof}

\begin{thm}\label{thm:ms_continuity}
	Let Assumption~\ref{ass2} hold.
	There is a $C>0$ such that for all $s,t\in\bT$
	\bee
	\bE\left(\|X(t)-X(s)\|^2_{V'}\right)\le C|t-s|.
	\eee
	Furthermore, if Assumption~\ref{ass2} holds with $\gg:=\min(\gg_0,\gg_F,\gg_G)>\frac{1}{2}$ and $\iota:=\min(\gg_0,\gg_F)\ge1$, then there is a $C>0$ such that for all $s,t\in\bT$
	\bee
	\bE\left(\|X(t)-X(s)\|^2_{H}\right)\le C|t-s|
	\quad\text{and}\quad 
	\|\bE(X(t)-X(s))\|_H\le C|t-s|.
	\eee
\end{thm}

\begin{proof}
	Under Assumption~\ref{ass2} it holds by Theorem~\ref{thm:exis} that
	\bee
	\sup_{t\in\bT} \|X(t)\|^2_{L^2(\gO;H)}\le C\left(1+\|X_0\|^2_{L^2(\gO;H)}\right)<+\infty.
	\eee
	With the weak formulation~\eqref{eq:ste_weak} we obtain for any test function $v\in V$ and $s,t\in\bT$ with $t\ge s$
	\be\label{eq:increment}
	(X(t)-X(s),v)_H=-\int_s^tB(X(r),v)dr+\int_s^t(F(r,X(r)),v)_Hds+
	\left(\int_s^tG(r,X(r))dL(r),v\right)_H.
	\ee
	Recalling that $\left\|\cdot\right\|_V = \left\|A^*\cdot\right\|_H$ yields with the triangle inequality
	\begin{align*}
		|(X(t)-X(s),v)_H|&\le \left|\int_s^t (X(r),A^*v)_Hdr\right|
		+\int_s^t \|F(r,X(r))\|_{V'}dr\|v\|_V
		+\Big\|\int_s^t G(r,X(r))dL(r)\Big\|_{V'}\|v\|_V\\
		&\le\left(\int_s^t \|X(r)\|_Hdr + \int_s^t \|F(r,X(r))\|_Hdr
		+\Big\|\int_s^t G(r,X(r))dL(r)\Big\|_H \right)\|v\|_V.
	\end{align*}
	Jensen's inequality then implies
	\begin{align*}
		\|X(t)-X(s)\|^2_{V'}
		&\le3\left(\left(\int_s^t \|X(r)\|_Hdr\right)^2
		+ \left(\int_s^t \|F(r,X(r))\|_Hdr\right)^2
		+\Big\|\int_s^t G(r,X(r))dL(r)\Big\|^2_H \right)\\
		&:=3(I^2+II^2+III^2).
	\end{align*}
	We take expectations and bound each term on the right and side.
	Since $\sup_{t\in\bT} \|X(t)\|_{L^2(\gO;H)}<+\infty$ it follows with Hölder's inequality that
	\bee
	\bE(I^2)\le \bE\left((t-s)\int_s^t \|X(r)\|_H^2dr\right)
	\le (t-s)^2 \sup_{t\in\bT} \|X(t)\|^2_{L^2(\gO;H)} \le C(t-s)^2.
	\eee
	The second term is bounded with Hölder's inequality and Assumption~\ref{ass2}\ref{item:Lipschitz2} by
	\be\label{eq:aux1}
	\begin{split}
		\bE(II^2)&\le (t-s)\int_s^t\bE(\|F(r,X(r))\|_H^2)dr \\
		&\le (t-s)\int_s^t 1+\bE(\|X(r)\|_H^2)dr \\
		&\le (t-s)^2 (1+\sup_{r\in\bT}\bE(\|X(r)\|_H^2))\\
		&\le C(t-s)^2,
	\end{split}
	\ee
	where we have used Theorem~\ref{thm:exis} in the last estimate.
	The last term $III$ is estimated by Lemma~\ref{lem:isometry} and Assumption~\ref{ass2}\ref{item:EV} and \ref{item:Lipschitz2}:
	\be\label{eq:aux2}
	\bE(III^2)= \int_s^t\bE(\|G(r,X(r))\|^2_{\LHS(\cU,H)})dr
	\le C\int_s^t1+\bE(\|X(r)\|^2_H)dr
	\le C(t-s).
	\ee
	The last inequality follows analogously to $\bE(II^2)$, which proves that
	\bee
	\bE\left(\|X(t)-X(s)\|^2_{V'}\right)\le C|t-s|.
	\eee
	
	To show mean-square continuity of $X$ with respect to $H$ we use that $V$ is dense in $H^\gg$ for $\gg<\frac{1}{2}$, see \cite[Theorem 3.4.3]{triebel2010theory}, and that $A\in \cL(H^\gg, H^{\gg-1})$ for $\gg>1/2$. 
	It follows by~\eqref{eq:increment} that 
	\be\label{eq:ms-weak}
	\begin{split}
		(X(t)-X(s),v)_H&
		=\int_s^t \dualpair{H^{\gg-1}}{H^{1-\gg}}{AX(r)}{v} dr+\int_s^t(F(r,X(r)),v)_Hds\\
		&\quad +\left(\int_s^tG(r,X(r))dL(r),v\right)_H,
	\end{split}
	\ee
	where $H^{-\gg}$ denotes the dual space of $H^\gg$ for $\gg>0$. Since $\gg>\frac{1}{2}$, we have that $V$ is dense in $H^{1-\gg}$ and $H$, and therefore find that 
	\begin{align*}
	\bE(\|X(t)-X(s)\|_H^2)&\le \int_s^t \bE(\|AX(r)\|_{H^{\gg-1}} \|X(t)-X(s)\|_{H^{1-\gg}}) dr \\
	&\quad +\int_s^t\bE(\|F(r,X(r))\|_H\|X(t)-X(s)\|_H) dr \\
	&\quad + \bE\left(\left\|\int_s^tG(r,X(r))dL(r)\right\|_H^2\right)^{1/2}\bE(\|X(t)-X(s)\|_H^2)^{1/2}\\
	&\le C 
	\left(\int_s^t \bE(\|X(r)\|_{H^\gg}^2)^{1/2}dr\right) \bE(\|X(t)-X(s)\|_{H^{1-\gg}}^2)^{1/2} \\
	&\quad +\left(\int_s^t\bE(\|F(r,X(r))\|_H^2)^{1/2}dr\right)
	\bE(\|X(t)-X(s)\|_H^2)^{1/2} \\
	&\quad + \bE\left(\left\|\int_s^tG(r,X(r))dL(r)\right\|_H^2\right)^{1/2}\bE(\|X(t)-X(s)\|_H^2)^{1/2}.
	\end{align*}
	We then obtain $\bE(\|X(t)-X(s)\|_H^2)\le C|t-s|$ with similar arguments as in the first part of the proof by Assumption~\ref{ass2}\ref{item:Lipschitz2}, Lemma~\ref{lem:isometry} and Theorem~\ref{thm:spatial_reg}.

	For the last part, taking expectations in~\eqref{eq:ms-weak} and testing against $v=\bE(X(t)-X(s))$ yields 
	\begin{align*}
		\|\bE(X(t)-X(s))\|_H^2&\le 
		\int_s^t \|\bE(AX(r))\|_H \|\bE(X(t)-X(s))\|_H dr \\
		&\quad +\int_s^t\|\bE(F(r,X(r)))\|_H\|\bE(X(t)-X(s))\|_H dr \\
		&\le C \left(\int_s^t \|\bE(X(r))\|_{H^1}dr\right) \|\bE(X(t)-X(s))\|_H \\
		&\quad +C\left(\int_s^t\|\bE(F(r,X(r)))\|_H dr\right) \|\bE(X(t)-X(s))\|_H \\
		&\le C |t-s|\|\bE(X(t)-X(s))\|_H,
	\end{align*}
	where we have used  Theorem~\ref{thm:spatial_reg} and that $\iota\ge 1$.
\end{proof}
%

In most cases, it is impossible to access $X$ analytically as the paths of $X$ are time-dependent random functions taking values in the infinite-dimensional Hilbert space $H$. The time dependency of each sample may be reflected in the coefficients of a suitable basis expansions, but in general no tractable representations are available.
Even if closed form solutions with respect to $X_0$ and a given path of $L$ were known, it would still be unclear how to sample the infinite-dimensional L\'evy process $L$.
We address these issues by introducing a discontinuous Galerkin spatial discretization in Section \ref{sec:dg} and suitable time stepping in Section \ref{sec:time}.
Thereafter, we discuss an approximate sampling technique for $L$, which yields a fully discrete approximation scheme for the stochastic transport problem.

\section{Discontinuous Galerkin Spatial Discretization}\label{sec:dg}

In this section we discretize Eq.~\eqref{eq:ste_weak} with respect to the spatial domain $\cD$.
To this end, let $\mathfrak H\subset (0,\infty)$ be a countable set and let $(\cK_h, h\in\mathfrak H)$ be a sequence of uniform, shape-regular triangulations of $\cD$, indexed by $h:=\max_{K\in\cK_h}\text{diam}(K)>0$. We denote by $\partial\cK_h$ the set of all faces of $\cK_h$ for fixed $h\in\mathfrak H$ . 
To achieve optimal spatial convergence rates, we follow the approach in \cite{CDG08} and consider triangulations $(\cK_h, h\in\mathfrak H)$ that satisfy certain \emph{flow conditions} with respect to transport direction $a$. The face $E\in \partial\cK_h$ of any given triangle $K$ is called an \emph{inflow face} if $a\cdot \vv n_K|_E>0$, and \emph{outflow face} whenever $a\cdot \vv n_K|_E<0$.
\begin{assumption}\label{ass:flow}
	The domain $\cD$ is polygonal with piecewise linear boundary (in case that $d=2$) and $(\cK_h, h\in\mathfrak H)$ is a sequence of uniform, shape-regular triangulations of $\cD$, such that $\mathfrak H\subseteq (0,1]$.
	For each $h\in\mathfrak H$, a given triangulation $\cK_h$ satisfies the following two conditions:
	\begin{align}
		\label{eq:flow_cond1}
		&\text{Each simplex $K\in\cK_h$ has a unique inflow face, denoted by $E_K^+$.}\\
		\label{eq:flow_cond2}
		&\text{Each interior inflow face  $E_K^+$ is included in an outflow face of another simplex $\widetilde K\in\cK_h$.} 
	\end{align}
\end{assumption}
We assume that $\cD$ is polygonal for convenience only, to omit errors due to the piecewise linear approximation of $\partial\cD$ in this case. The ensuing numerical analysis shows that the flow conditions on the inflow/ouflow faces on $\cK_h$ are indeed crucial to obtain the maximum rate of convergence for $d=2$. In case that $d=1$, \eqref{eq:flow_cond1} and~\eqref{eq:flow_cond2} are always satisfied when partitioning $\cD$ into intervals of arbitrary length.

The discrete discontinuous Galerkin (DG) space $H_h\subset H$ of piecewise linear polynomials is given by
\bee
H_h:=\left\{w\in H\big|\; w|_K\in\mathbf{P_1}(K),\;K\in\cK_h\right\}.
\eee
Elements of $H_h$ are piecewise linear on the simplices $K$, but allow for jumps at the interfaces of the triangulation.
Hence, the space of continuous, piecewise linear finite elements is enclosed by $H_h$, whereas the DG approach yields additional stability with a suitably chosen numerical flux on the discontinuities.

Note that $Aw_h\notin L^2(\cD)$ for $w_h\in H_h$, but in general $Aw_h\in V'$, which is in line with the distributional extension $A:H\to V'$ from Section~\ref{sec:ste}.
The "broken version" of the $H$-scalar product and the induced norm with respect to $\cK_h$ are denoted by 
\be\label{eq:broken_norm}
(v,w)_{L^2(\cK_h)}:=\sum_{K\in\cK_h}(v,w)_{L^2(K)},\quad \|v\|_{L^2(\cK_h)}:=(v,v)_{L^2(\cK_h)}.
\ee
Clearly, $(\cdot,\cdot)_{L^2(\cK_h)}$ and $(\cdot,\cdot)_H$ coincide on $H$.
Analogously, we define the broken Sobolev norms and semi-norms for any $\gg>0$ via
\bee
\|v\|_{H^\gg(\cK_h)}:=\sum_{K\in\cK_h}\|v\|_{H^\gg(K)},\quad |v|_{H^\gg(\cK_h)}:=\sum_{K\in\cK_h}|v|_{H^\gg(K)}.
\eee
The corresponding broken Sobolev spaces are given as
\bee
H^\gg(\cK_h):=\{v\in H|\; \|v\|_{H^\gg(\cK_h)}<\infty\},
\eee
and we note that $H_h\subset H^1(\cK_h)$.
The key ingredient to prove optimal convergence rates for deterministic transport equations in \cite{CDG08} is a special projection $\cP_h:H\to H_h$, defined by the two properties
\begin{equation}\label{eq:projection}
	\begin{split}
		(\cP_hw-w, v)_{K} &= 0\quad \text{for any $K\in\cK_h, w\in H$ and $v\in \mathbf{P_0}(K)$,}\\
		(\cP_hw-w, \widehat v)_{E_K+} &= 0\quad \text{for any $K\in\cK_h, w\in H$ and $\widehat v\in \mathbf{P_1}(E_K^+)$.}\\
	\end{split} 
\end{equation}

\begin{lem}{\cite[Proposition 2.1]{cockburn2008superconvergent}/\cite[Lemma 2.1]{CDG08}}
	\label{lem:DG_interp}
	Under Assumption~\ref{ass:flow}, the projection $\cP_h:H\to H_h$ given by~\eqref{eq:projection} is well-defined for any fixed $h\in\mathfrak H$. 
	Furthermore, for any $\gg\ge 0$ there is a $C>0$ such that for any $v\in H^\gg(\cK_h)$ and $h\in\mathfrak H$ there holds
	\bee
	\|v-\cP_hv\|_{L^2(\cK_h)}\le C|v|_{H^\gg(\cK_h)}\,h^{\min(2,\gg)}.
	\eee
\end{lem}

In order to derive a discrete weak formulation for~\eqref{eq:ste} with respect to $H_h$, we define the bilinear form
\begin{equation}\label{eq:blf_discrete}
	B_h:H^1(\cK_h)\times H^1(\cK_h)\to\bR,\quad (w_h,v_h)\mapsto (w_h, a\cdot \nabla v_h)_{L^2(\cK_h)} 
	-  \sum_{E\in\partial\cK_h\setminus\cD^+}(\vv n\cdot a w_h, v_h)_{L^2(E)}.
\end{equation}
Note that the homogeneous Dirichlet boundary conditions at $\partial\cD^+$ have been imposed implicitly~\eqref{eq:blf_discrete}, as we do not sum over the inflow boundary edges.
As $aw_h$ is not uniquely determined at the faces $E$, it remains to determine a \textit{numerical flux} on each face, that is not parallel to the flow direction $a$ or belongs to the outflow boundary $\partial\cD^{-}$. 
To this end, consider two given simplices $K^+,K^-$ sharing a common interior face $E\in\partial\cK_h$ such that $E\cap\partial\cD=\emptyset$.
The outward normal vectors of $K^+$ and $K^-$ on $E$ are denoted by $\vv n^+$ and $\vv n^-$, respectively.
Similarly, for a scalar function $\psi:K^+\cup K^-\to\bR$, we define by $\psi^+$ the trace of $\psi|_{K^+}$ on $E$, and $\psi^-$ is the trace of $\psi|_{K^-}$ on $E$.
We denote the jump $\dgj \cdot$ resp. average $\dga\cdot$ across $E$ of $\psi$ by
\bee
\dgj{\psi}
:=
	\vv n^+\psi^++\vv n^-\psi^-
\quad\text{and}\quad\dga{\psi}:=\frac{\psi^+ + \psi^-}{2}.
\eee
Note that $\dgj{\psi}$ is a vector for scalar $\psi$, and that conversely $\dgj{\psi}$ is scalar if $\psi$ is a vector-valued function.
As numerical flux on $E$, we use the \textit{upwind flux} given by
\be\label{eq:upwind_flux}
(\vv n\cdot a w_h, v_h)_{L^2(E)}
:=\int_E aw_h^-\cdot\dgj{v_h}dz
=\int_E \left(\dga{aw_h}-\frac{|a\cdot \vv n|}{2}\dgj{w_h}\right)\cdot\dgj{v_h}dz,\quad w_h,v_h\in H_h
\ee
(for the second equality in \eqref{eq:upwind_flux} see, e.g., \cite[Section 3]{BMS04}).
To treat the boundary edges $E\subset\partial\cD^-$, we define the numerical flux
\bee
(\vv n\cdot a w_h,v_h)_{L^2(E)}:=\int_E aw_h^-\cdot\dgj{v_h}dz
=-\int_E a\cdot \vv n^{+} w_h^-v_h^-dz,
\eee
with the convention that $\vv n^+$ is the \emph{outward} pointing normal of $\cD$ at $E$.

For simplicity, we set the discrete initial condition as $X_h(0):=\cP_hX_0$. The semi-discrete weak problem is now to find $X_h:\gO\times\bT\to H_h$, such that for any $t\in\bT$ and  $v_h\in H_h$ it holds that
\be
\begin{split}\label{eq:ste_dg}
	(X_h(t),v_h)_H +\int_0^t B_h(X_h(s),v_h) ds 
	&=
	(X_h(0),v_h)_H
	+\int_0^t (F(s, X_h(s)), v_h)_Hds \\
	&\quad +\left(\int_0^t G(s,X_h(s))dL_s,v_h \right)_H.
\end{split}
\ee

As one of the main results of this article, we provide an estimate on the semi-discrete error $X-X_h$: 
\begin{thm}\label{thm:DGerror}
	Let Assumptions~\ref{ass2} and~\ref{ass:flow} hold such that $\gg=\min(\gg_0, \gg_F, \gg_G)>1/2$, and denote by $X$ and $X_h$ (for any $h\in\mathfrak H$) the solutions to Problem~\eqref{eq:ste_weak} and Problem~\eqref{eq:ste_dg}, respectively.
	Then, there is a constant $C>0$ such that for all $h\in\mathfrak H$ and $t\in\bT$ there holds
	\bee
	\sup_{t\in \bT} \|X(t)-X_h(t)\|_{L^2(\gO; H)}\le C h^{\min(\gg,2)}.
	\eee
\end{thm}

As a first step for our error analysis, we have to make sure that the true solution $X$ is actually in line with the weak formulation in~\eqref{eq:ste_dg}. As $H_h\subsetneq V$, this is not obvious at first glance. 
\begin{lem}\label{lem:weak_conform}
	Let Assumption~\ref{ass2} hold such that $\gg=\min(\gg_0,\gg_F,\gg_G)>\frac{1}{2}$.
	Then, the unique weak solution $X$ to~\eqref{eq:ste_weak} satisfies for all $t\in[0,T]$ and $v_h\in H_h$ that
	\be
	\begin{split}\label{eq:ste_dg_conform}
		(X(t),v_h)_H +\int_0^t B_h(X(s),v_h) ds 
		&=
		(X(0),v_h)_H
		+\int_0^t (F(s, X(s)), v_h)_Hds \\
		&\quad +\left(\int_0^t G(s,X(s))dL_s,v_h \right)_H.
	\end{split}
	\ee
\end{lem}

\begin{proof}
	Theorem~\ref{thm:spatial_reg} yields that $X(t)\in H^\gg$ for all $t\in[0,T]$ and $\gg>1/2$. 
	This implies in particular, that the trace $X(t)|_{\partial K}\in L^2(K)$ is well-defined for any $K\in\cK_h$ and coincides with the numerical flux on every face $E\subset\partial K$, in the sense that there holds on each face $E\in\partial\cK_h$
	\begin{equation}\label{eq:num_flux}
		(\vv n\cdot aX(t), v_h)_{L^2(E)} = 
		\int_{E} \left(\dga{aX(t)}-\frac{|a\cdot \vv n|}{2}\dgj{X(t)}\right)\cdot\dgj{v_h}dz,\quad v_h\in H_h.
	\end{equation}
	
	Now assume first that $X(t)\in D(A)$ for any $t\in[0,T]$ and fix an arbitrary test function $v_h\in H_h\subset H$.
	By density of $V=A^{-*}H$ in $H$, partial integration yields on any $K\in\cK_h$ that
	\begin{equation}\label{eq:simplex}
		\begin{split}
			&(X(t),v_h)_{L^2(K)} - \int_0^t(AX(s), v_h)_{L^2(K)} ds \\
			&=
			(X(t),v_h)_{L^2(K)} + \int_0^t (X(s), a\cdot\nabla v_h)_{L^2(K)} 
			- (\vv n\cdot aX(s), v_h)_{L^2(\partial K)}ds\\
			&=
			(X(0),v_h)_{L^2(K)}
			+\int_0^t (F(s, X(s)), v_h)_{L^2(K)}ds 
			+\left(\int_0^t G(s,X(s))dL_s,v_h \right)_{L^2(K)}.
		\end{split}
	\end{equation} 
	Summation with respect to $K\in\cK_h$ thus yields~\eqref{eq:ste_dg_conform}.
	
	If $X(t)\notin D(A)$, we consider a sequence $(X_n, n\in\bN)$ of smooth approximations to $X$ instead, such that $X_n(t)\to X(t)$ in $H^\gg$ for all $t\in[0,T]$.
	Since we have for all $s\in[0,T]$ that
	\begin{equation*}
		\lim_{n\to\infty} \abs{(\vv n\cdot a(X(s)-X_n(s)), v_h)_{L^2(\partial K)}}
		\le \lim_{n\to\infty} C \|X(s)-X_n(s)\|_{H^{\gg}(K)}\|v_h\|_{H^{\gg}(K)} = 0
	\end{equation*}
	by the trace theorem, there holds
	\begin{align*}
		-\lim_{n\to\infty}\int_0^t(AX_n(s), v_h)_{L^2(K)} ds
		&=
		\lim_{n\to\infty} \int_0^t (X_n(s), a\cdot\nabla v_h)_{L^2(K)} 
		-(\vv n\cdot aX_n(s), v_h)_{L^2(\partial K)}ds \\
		&=
		\int_0^t (X(s), a\cdot\nabla v_h)_{L^2(K)} 
		- (\vv n\cdot aX(s), v_h)_{L^2(\partial K)}ds.
	\end{align*} 
	The claim thus follows by replacing $X$ by $X_n$ in~\eqref{eq:simplex} for given $K\in\cK_h$, taking the limit $n\to\infty$ and summing over all $K\in\cK_h$.
\end{proof}

To prove~\ref{thm:DGerror}, we also need to establish that $B_h$ is positive semi-definite.
\begin{lem}\label{lem:blf_pos}
	For any $v_h\in H^1(\cK_h)$ it holds that
	\begin{equation*}
			B_h(v_h,v_h)
			=\sum_{E\in\partial\cK_h}\int_E\frac{|a\cdot n|}{2}\dgj{v_h} \cdot\dgj{v_h}dz
			\ge 0.
		\end{equation*}
\end{lem}

\begin{proof}
	By~\eqref{eq:blf_discrete}, it holds for $v_h\in H^1(\cK_h)$ that
	\begin{equation}\label{eq:blf1}
		B_h(v_h,v_h) = (v_h, a\cdot \nabla v_h)_{L^2(\cK_h)} 
		-  \sum_{E\in\partial\cK_h\setminus\cD^+}(\vv n\cdot a v_h, v_h)_{L^2(E)}.
	\end{equation}
	For any interior edge $E\in\partial\cK_h$ we have by~\eqref{eq:upwind_flux}
	\bee
	\int_E\left(\dga{av_h}-\frac{|a\cdot n|}{2}\dgj{v_h}\right)\cdot\dgj{v_h}dz
	=\int_E\frac{a}{2}\cdot\dgj{v_h^2}-\frac{|a\cdot n|}{2}\dgj{v_h} \cdot\dgj{v_h}dz,
	\eee
	and the second term in~\eqref{eq:blf1} then reads
	\be\label{eq:blf2}
	\begin{split}
	 	\sum_{E\in\partial\cK_h\setminus\cD^+}(\vv n\cdot a v_h, v_h)_{L^2(E)}
		&= \sum_{E\in\partial\cK_h, E\cap\partial\cD=\emptyset}
		\int_E\frac{a}{2}\cdot\dgj{v_h^2}-\frac{|a\cdot n|}{2}\dgj{v_h} \cdot\dgj{v_h}dz\\
		&\quad+ \sum_{E\in\partial\cK_h, E\subset\partial\cD^-}\int_E a\cdot\vv n^{+} (v_h^{-})^2dz.
	\end{split}
	\ee
	On the other hand, partial integration yields
	\begin{align*}
		(v_h,a\cdot \nabla v_h)_{L^2(\cK_h)}
		&=-(a\cdot \nabla  v_h,v_h)_{L^2(\cK_h)}+\sum_{K\in\cK_h}\int_{\partial K} (\vv n\cdot a)v_h^2dz,
	\end{align*}
	and thus the first term in~\eqref{eq:blf1} is given by 
	\begin{equation}\label{eq:blf3}
		(v_h,a\cdot \nabla v_h)_{L^2(\cK_h)}
		=\frac{1}{2}\left(\sum_{E\in\partial\cK_h, E\cap\partial\cD=\emptyset}
		\int_E a\cdot\dgj{v_h^2}dz + 
		\sum_{E\in\partial\cK_h, E\subset\partial\cD}\int_E a\cdot \vv n^{+} (v_h^-)^2dz\right).
	\end{equation}
	Substituting Eqs.~\eqref{eq:blf2} and~\eqref{eq:blf3} in~\eqref{eq:blf1} now shows the claim, since
	\begin{equation*}
			B_h(v_h,v_h)
			=\sum_{E\in\partial\cK_h}\int_E\frac{|a\cdot n|}{2}\dgj{v_h} \cdot\dgj{v_h}dz\ge 0.
		\end{equation*}
\end{proof}

We are now ready to prove the main result of this section.
	
\begin{proof}[Proof of Theorem~\ref{thm:DGerror}]
	We define the stochastic processes $e$, $\phi$ and $\psi$ via $e(t):=X(t)-X_h(t)$, $\phi(t):=(1-\cP_h)X(t)$ and $\psi(t):=\cP_hX(t)-X_h(t)$ for $t\in\bT$, so that $e(t) = \phi(t)+\psi(t)$.
	Theorem~\ref{thm:spatial_reg} and Lemma~\ref{lem:DG_interp} show that
	\begin{equation}\label{eq:phi-est}
		\sup_{t\in \bT}\|\phi(t)\|_{L^2(\gO;H)}\le C h^{\min(\gg,2)},
	\end{equation}
	and it thus remains to derive the corresponding bound for $\psi$.
	
	Lemma~\ref{lem:weak_conform} and \eqref{eq:ste_dg} yield together with $X_h(0)=\cP_hX_0$ for all $v_h\in H_h$ that 
	\bee
	\begin{split}
		(e(t),v_h)_H+\int_0^tB_h(\phi(s)+\psi(s),v_h)ds
		&=-(\phi(0),v_h)_H
		+\int_0^t(F(s,X(s))-F(s,X_h(s)),v_h)_Hds\\
		&\quad+\left(\int_0^tG(s,X(s))-G(s,X_h(s))dL(s),v_h\right)_H.
	\end{split}
	\eee
	We now follow the proof of \cite[Theorem 2.2]{CDG08} to show that $B_h(\phi(s),v_h)=0$. For any $s\in[0,T]$, it holds by~\eqref{eq:blf_discrete} that 
	\begin{align*}
		B_h(\phi(s),v_h) 
		&= (\phi(s), a\cdot\nabla v_h)_{L^2(\cK_h)} -
		\sum_{E\in\partial\cK_h\setminus\cD^+}(\vv n\cdot a \phi(s), v_h)_{L^2(E)} \\
		&= -\sum_{E\in\partial\cK_h\setminus\cD^+}(\vv n\cdot a \phi(s), v_h)_{L^2(E)} \\
		&= - \sum_{E\in\partial\cK_h\setminus\cD^+}\int_E a \phi^{-}(s)\cdot \dgj{v_h} dz \\
		&= - \sum_{E\in\partial\cK_h\setminus\cD^+}\int_E \phi(s) 
		a\cdot\left(\vv n^+v_h^+ + \vv n^-v_h^-\right) dz \\
		&=0. 
	\end{align*}
	The second and last identity follow by the defining properties of $\cP_h$ in~\eqref{eq:projection}, in the third and fourth step we have used the definition of the numerical flux in~\eqref{eq:num_flux}.
	Thus, we obtain the error equation
	\be\label{eq:psi}
	\begin{split}
		(\psi(t),v_h)_H
		&=(\phi(t)-\phi(0),v_h)_H-\int_0^tB_h(\psi(s),v_h)ds
		+\int_0^t(F(s,X(s))-F(s,X_h(s)),v_h)_Hds\\
		&\quad+\left(\int_0^tG(s,X(s))-G(s,X_h(s))dL(s),v_h\right)_H.
	\end{split}
	\ee
	It\^o's formula from Lemma~\ref{lem:ito-formula} yields for $Y=\psi$ and $\Psi(v):=\|v\|_H^2$ that 
	\bee
	\begin{split}
		\|\psi(t)\|_H^2
		&=(\phi(t)-\phi(0),\psi(t))_H-2\int_0^tB_h(\psi(s),\psi(s))ds
		+2\int_0^t(F(s,X(s))-F(s,X_h(s)),\psi(s))_Hds\\
		&\quad+2\left(\int_0^t[G(s,X(s))-G(s,X_h(s))]^*\psi(s), dL(s)\right)_\cU
		+\int_0^t\cI d[\psi,\psi]_s 
	\end{split}
	\eee
	where the operator $\cI\in\cL(H\widehat \otimes H, \bR)$ is given by $\cI(v\otimes w):=(v,w)_H$.
	Note that the "jump terms" in the second and third line in Eq.~\eqref{eq:ito_formula} vanish for the functional $\Psi(v)=\|v\|_H^2$.
	
	By Lemma~\ref{lem:blf_pos} we have that $B_h(\psi(s),\psi(s))\ge 0$, taking expectations thus yields  the estimate 
	\be\label{eq:psi-weak}
	\begin{split}
		\bE(\|\psi(t)\|_H^2)
		&\le\bE((\phi(t)-\phi(0),\psi(t))_H) \\
		&\quad +2\bE\left(\int_0^t(F(s,X(s))-F(s,X_h(s)),\psi(s))_Hds\right) 
		+\bE\left(\int_0^t\cI d[\psi,\psi]_s\right) \\
		&:=I+2II+III.
	\end{split}
	\ee
	The first term $I$ is bounded by Young's inequality and~\eqref{eq:phi-est} by
	\bee
		I\le \frac{1}{2}\bE(\|\phi(t)-\phi(0)\|_H^2+\|\psi(t)\|_H^2) 
		\le C h^{2\min(\gg,2)}+\frac{1}{2}\bE(\|\psi(t)\|_H^2).
	\eee
	To bound the second term $II$, we use Assumption~\ref{ass2}~\ref{item:Lipschitz2} and~\eqref{eq:phi-est} to obtain
	\begin{align*}
		II
		&\le \frac{1}{2} \int_0^t \bE(\|F(s,X(s))-F(s,X_h(s)\|_H^2) + \bE(\|\psi(s)\|_H^2) ds\\
		&\le C \int_0^t \bE(\|X(s)-X_h(s)\|_H^2) + \bE(\|\psi(s)\|_H^2) ds \\
		&\le C \int_0^t \bE(\|\phi(s)+\psi(s)\|_H^2) + \bE(\|\psi(s)\|_H^2) ds \\
		&\le C\left(h^{2\min(\gg,2)} + \int_0^t \bE(\|\psi(s)\|_H^2) ds\right). 
	\end{align*}
	For the final term $III$, we first observe that the quadratic variation term is given for any orthonormal basis $(f_i, i\in\bN)$ of $H$ by 
	\begin{align*}
		\int_0^t\cI d[\psi,\psi]_s 
		&= \int_0^t\cI d \left(\sum_{i,j\in\bN} f_i\otimes f_j \left[(\psi, f_i)_H, (\psi, f_j)_H \right]_s\right)\\
		&= \int_0^t\sum_{i,j\in\bN}^\infty (f_i, f_j) d \left[(\psi, f_i)_H, (\psi, f_j)_H \right]_s \\
		&= \sum_{i\in\bN} \left[(\psi, f_i)_H, (\psi, f_i)_H \right]_t. 
	\end{align*}
	Note that for all $i\in\bN$ there holds
	\begin{align*}
		(\psi, f_i)_H
		&=
		\left(\int_0^t[\cP_hG(s,X(s))-G(s,X_h(s))]^*f_i, dL(s) \right)_\cU \\
		&= 
		\sum_{k\in\bN}\int_0^t\left([\cP_hG(s,X(s))-G(s,X_h(s))]^*f_i, \sqrt{\eta_k} e_k \right)_\cU d\ell_k(s), 
	\end{align*}
	and all terms in the second line are centered and jointly uncorrelated random variables with respect $k$.
	As $(\sqrt{\eta_k}e_k,k\in\bN)$ is an orthonormal basis of $\cU$, we obtain by Parseval's identity
	\begin{align*}
		III &= \left(\int_0^t\cI d[\psi,\psi]_s\right) \\
		&=  \int_0^t \bE\left( \sum_{i\in\bN} 
		\left\|[\cP_hG(s,X(s))-G(s,X_h(s))]^*f_i\right\|_\cU^2 \right)ds\\
		&=  C \int_0^t \bE\left(
		\left\|[\cP_hG(s,X(s))-G(s,X_h(s))]\right\|_{\cL_{HS}(\cU;H)}^2 \right)ds\\
		&\le C \int_0^t \bE\left(
		\left\|[G(s,X(s))-\cP_hG(s,X(s))]\right\|_{\cL_{HS}(\cU;H)}^2 \right)
		+
		\bE\left(
		\left\|[G(s,X(s))-G(s,X_h(s))]\right\|_{\cL_{HS}(\cU;H)}^2 \right)ds\\
		&\le  C \int_0^t h^{2\min(\gg, 2)}\bE\left(\|G(s,X(s))\|_{\cL_{HS}(\cU;H^\gg)}^2 \right)
		+
		\bE\left(\|[X(s)-X_h(s)]\|_H^2 \right)ds,
	\end{align*}
	where the last estimate holds by Lemma~\ref{lem:DG_interp} and Assumption~\ref{ass2}~\ref{item:Lipschitz2}.
	By Item~\ref{item:lingrowthG} in Assumption~\ref{ass2} and Theorem~\ref{thm:spatial_reg}, we obtain further for some $\ga>1$ and $\gb\in(0,(\ga-1)/2\ga)$ that 
	\begin{align*}
		\bE\left(\|G(s,X(s))\|_{\cL_{HS}(\cU;H^\gg)}^2\right) 
		&=\sum_{k\in\bN} \eta_k\|G(s,X(s))e_k\|_H^2 \\
		&\le C(1+\bE(\|X(s)\|_{H^\gg}^2)\sum_{k\in\bN} \eta_k^{1-2\gb} \\
		&\le \sum_{k\in\bN} k^{-\ga(1-2\gb)} \\
		&< \infty.
	\end{align*}
	Consequently, by~\eqref{eq:phi-est} there holds
	\begin{align*}
		III \le C\left(h^{2\min(\gg,2)} + \int_0^t \bE(\|\psi(s)\|_H^2) ds\right).
	\end{align*}
	Substituting the estimates on $I-III$ back into~\eqref{eq:psi-weak} thus yields 
	\bee
		\bE(\|\psi(t)\|_H^2)
		\le C\left(h^{2\min(\gg,2)} + \int_0^t \bE(\|\psi(s)\|_H^2) ds\right), 
	\eee
	and the claim follows by Grönwall's inequality.
\end{proof}

\section{Temporal Discretization}\label{sec:time}

We use $(m+1)$ equidistant time points $0=t_0<\dots<t_m=T$ to discretize $\bT$, and define $\gD t:=T/m>0$, for $m\in\bN$. We  employ a \textit{backward Euler} (BE) approximation for the linear part of Eq.~\eqref{eq:ste_dg}, i.e.
\bee
\int_{t_{i-1}}^{t_i}B_h(X_h(s),v_h)ds\approx\gD t B_h(X_h(t_i),v_h),\quad i=1,\dots,m.
\eee
The nonlinear parts with respect to $F$ and $G$ are approximated by the forward differences
\be\label{eq:time_fw}
\begin{split}
	\int_{t_{i-1}}^{t_i}(F(s,X_h(s)),v_h)_Hds&\approx (F(t_{i-1},X_h(t_{i-1}))\gD t,v_h)_H,\\
	\left(\int_{t_{i-1}}^{t_i}G(s,X_h(s))dL(s),v_h\right)_H&\approx(G(t_{i-1},X_h(t_{i-1}))\gD L^{(i)},v_h)_H,
\end{split}
\ee
where $\gD L^{(i)}:=L(t_i)-L(t_{i-1})$.
As the stochastic integral on the left hand side in Eq.~\eqref{eq:time_fw} is an It\^o integral, it is crucial to use forward differences that preserve the martingale property of the driving noise.
For the nonlinearity $F$ on the other hand, we could have chosen a backward difference or midpoint rule, but with the scheme~\eqref{eq:time_fw} we avoid solving a nonlinear system in every time step without affecting the overall order of convergence.
The spatio-temporal-discrete version of the weak problem is then to find $(X_h^{(i)},i=0,\dots,m)\subset H_h$ such that $X^{(0)}=X_0$ and for any $v_h\in H_h$ and $i=1,\dots,m$
\be \label{eq:ste_weak_td}
(X_h^{(i)}-X_h^{(i-1)},v_h)_H+\gD tB_h(X_h^{(i)},v_h)=\gD t(F(t_{i-1},X_h^{(i-1)}),v_h)_H+(G(t_{i-1},X_h^{(i-1)})\gD L^{(i)},v_h)_H.
\ee
We further interpolate $X_h^{(\cdot)}$ linearly at the temporal nodes to obtain the continuous time approximation $\ol X_h:\bT\to H_h$ defined by
\begin{equation}\label{eq:ste_weak_int}
	\ol X_h(t):=X_h^{(i)}+\frac{t-t_i}{\gD t}(X_h^{(i+1)}-X_h^{(i)}),
	\quad t\in[t_i,t_{i+1}],\quad i=0,\dots, m-1.
\end{equation} 

Our main result of this section bounds the resulting temporal discretization error:

\begin{thm}\label{thm:time_error}
	Let Assumptions~\ref{ass2} and~\ref{ass:flow} hold such that $\gg=\min(\gg_0, \gg_F, \gg_G)\ge1$, let $\gD t\in(0,1]$, and denote by $X_h$ and $\ol X_h$ the solutions to Problem~\eqref{eq:ste_dg} and Problem~\eqref{eq:ste_weak_int}, respectively.
	Then, there is a constant $C>0$ such that for all $h\in\mathfrak H$ with $h^2\ge \gD t$ and $t\in\bT$ there holds
	\bee
	\sup_{t\in\bT}\|X_h(t)-\ol X_h(t)\|_{L^2(\gO; H)} \le C \gD t^{1/2}.
	\eee
\end{thm}

To prove Theorem~\ref{thm:time_error}, we first show mean-square continuity of the semi-discrete solution $X_h$. For this, we record in turn the well-known inverse estimate 
\begin{equation}\label{eq:inverse_est}
	\|v_h\|_{H^1(\cK_h)}+h^{-1/2}\|v_h\|_{L^2(\partial\cK_h)}\le C h^{-1} \|v_h\|_H, \quad v_h\in H_h, 
\end{equation}
where $C$ is independent of $h$ and $v_h$ (see, e.g., \cite[Remark 4.8/Corollary 4.24]{cangiani2022hp}).

\begin{lem}\label{lem:ms-cont-disc}
	Let Assumptions~\ref{ass2} and~\ref{ass:flow} hold such that $\gg=\min(\gg_0, \gg_F, \gg_G)\ge 1$, and denote by $X_h$ the solution to Problem~\eqref{eq:ste_dg}.
	Then, there is a constant $C>0$ such that for all $h\in\mathfrak H$ and $t,s\in\bT$ with $|t-s|\le 1$ there holds
	\begin{align*}
	\|X_h(t)-X_h(s)\|_{L^2(\gO; H)} \le C |t-s|^{1/2}
	\quad \text{and}\quad
	\|\bE(X_h(t)-X_h(s))\|_H \le C|t-s|.
	\end{align*}
\end{lem}

\begin{proof}
	We use again the processes $e$, $\phi$ and $\psi$ given by $e(t):=X(t)-X_h(t)$, $\phi(t):=(1-\cP_h)X(t)$ and $\psi(t):=\cP_hX(t)-X_h(t)$ for $t\in\bT$. 
	Further, without of loss of generality, we assume that $\gg\le 2$ and show mean-square continuity of $\psi$, the claim for $X_h$ then follows by the triangle inequality.
	
	Fix $t,s\in\bT$ such that $t> s$ and recall from~\eqref{eq:psi} that for $v_h\in V_h$ there holds
	\be\label{eq:psi2}
	\begin{split}
		(\psi(t)-\psi(s),v_h)_H
		&=(\phi(t)-\phi(s),v_h)_H-\int_s^tB_h(\psi(r),v_h)dr
		+\int_s^t(F(r,X(r))-F(r,X_h(r)),v_h)_Hdr\\
		&\quad+\left(\int_s^tG(r,X(r))-G(r,X_h(r))dL(r),v_h\right)_H.
	\end{split}
	\ee
	For the first part of the claim, we test against $v_h=\psi(t)-\psi(s)$ to obtain 
	\begin{align*}
		\|\psi(t)-\psi(s)\|_H^2
		&=(\phi(t)-\phi(s),\psi(t)-\psi(s))_H
		- \int_s^tB_h(\psi(r),\psi(t)-\psi(s))ds\\
		&\quad+\int_s^t(F(r,X(r))-F(r,X_h(r)),\psi(t)-\psi(s))_Hds\\
		&\quad+\left(\int_s^tG(r,X(r))-G(r,X_h(r))dL(r),\psi(t)-\psi(s)\right)_H\\
		&:=I+II+III+IV.
	\end{align*}
	The first term is bounded in expectation by Theorem~\ref{thm:ms_continuity} 
	\begin{align*}
		\bE(I)\le C \bE(\|(1-\cP_h)(X(t)-X(s))\|_H^2)^{1/2}\bE(\|\psi(t)-\psi(s)\|_H^2)^{1/2}
		\le C |t-s|^{1/2}\bE(\|\psi(t)-\psi(s)\|_H^2)^{1/2}.
	\end{align*}
	To bound the second term, we use~\eqref{eq:projection} together with the inverse estimate~\eqref{eq:inverse_est}, Lemma~\ref{lem:DG_interp}, and Theorem~\ref{thm:DGerror} to obtain 
	\begin{align*}
		\bE(II)
		&= \int_s^t\bE( (\psi(r),a\cdot \nabla (\psi(t)-\psi(s)))_{L^2(\cK_h)} )dr
		+ \int_s^t\sum_{E\in\partial\cK_h\setminus\cD^+} 
		\bE((\vv n\cdot a \psi(r),\psi(t)-\psi(s))_{L^2(E)})ds\\
		&\le C\left(\int_s^t h^{-1}\bE(\|\psi(r)\|_H^2)^{1/2}ds\right)
		\,\bE(\|\psi(t)-\psi(s)\|_H^2)^{1/2}\\
		&\le Ch^{\gg-1} |t-s|\,\bE(\|\psi(t)-\psi(s)\|_H^2)^{1/2}\\
		&\le C |t-s|\,\bE(\|\psi(t)-\psi(s)\|_H^2)^{1/2}.
	\end{align*}
	The last line follows since $h\le 1$ by Assumption~\ref{ass:flow} and with $\gg\ge 1$.
	For the last two terms $III$ and $IV$, we proceed as in the proof of Theorem~\ref{thm:DGerror} to obtain that   
	\begin{align*}
		\bE(III)&
		\le C \int_s^t\bE(\|\psi(r)\|_H^2))^{1/2}ds
		\,\bE(\|\psi(t)-\psi(s)\|_H^2)^{1/2} \\
		&\le Ch^\gg|t-s|\,\bE(\|\psi(t)-\psi(s)\|_H^2)^{1/2} \\
		&\le C|t-s|\,\bE(\|\psi(t)-\psi(s)\|_H^2)^{1/2}
	\end{align*}
	and
	\begin{align*}
		\bE(IV)&\le C \left(\int_s^t\bE(\|\psi(r)\|_H^2)ds\right)^{1/2}
		\,\bE(\|\psi(t)-\psi(s)\|_H^2)^{1/2} \\
		&\le Ch^\gg|t-s|^{1/2}\,\bE(\|\psi(t)-\psi(s)\|_H^2)^{1/2}\\
		&\le C|t-s|^{1/2}\,\bE(\|\psi(t)-\psi(s)\|_H^2)^{1/2}.
	\end{align*}
	As we have $|t-s|\le 1$, there holds 
	\begin{align*}
		\bE(\|\psi(t)-\psi(s)\|_H^2) \le C|t-s|^{1/2},
	\end{align*}
	and the first part of the claim follows with $e=\phi+\psi$, since
	\begin{equation*}
		\bE(\|X_h(t)-X_h(s)\|_H^2) 
		\le 
		\bE(\|X(t)-X(s)\|_H^2) + \bE(\|e(t)-e(s)\|_H^2)
		\le C|t-s| + \bE(\|\psi(t)-\psi(s)\|_H^2).
	\end{equation*}

	To proof the second part of the claim, we first take expectations in~\eqref{eq:psi2}, so that the stochastic integral vanishes, and then test against $v_h=\bE(X_h(t)-X_h(s))$. 
	With the Cauchy-Schwarz inequality, Theorem~\ref{thm:spatial_reg} and similar arguments as for the first part, we arrive at 
	\begin{align*}
		\|\bE(\psi(t)-\psi(s))\|_H^2
		&= \left(\bE(\phi(t)-\phi(s)), \bE(\psi(t)-\psi(s))\right)_H
		- \int_s^tB_h(\bE(\psi(r)),\bE(\psi(t)-\psi(s)))dr \\
		&\quad+\int_s^t(\bE(F(r,X(r))-F(r,X_h(r))),\bE(\psi(t)-\psi(s)))_Hdr\\
		&\le C\left(|t-s|+h^{\gg-1} |t-s|+h^\gg|t-s|\right)\,\|\bE(\psi(t)-\psi(s))\|_H\\
		&\le C|t-s|\|\bE(\psi(t)-\psi(s))\|_H,
	\end{align*}
	and the second part of the claim follows again by the triangle inequality.
\end{proof}

\begin{proof}[Proof of Theorem~\ref{thm:time_error}]
	Since $\ol X_h$ is a convex combination of $X_h^{(\cdot)}$ in each sub-interval, it suffices to bound the error at the temporal grid points. 
	We assume again that $\gg\le 2$ for convenience.
	Let $\phi^{(i)}:=X_h(t_i)-X_h^{(i)}$ for $i\in\{0, 1,\dots,m\}$ and observe that 
	$\phi^{0}=0$. 
	By Eqs.~\eqref{eq:ste_dg} and \eqref{eq:ste_weak_td} it holds for all $v_h\in H_h$ that
	\be
	\begin{split}\label{eq:time_error1}
		(\phi^{(i)}-\phi^{(i-1)},v_h)_H+\gD t B_h(\phi^{(i)},v_h)
		&= \int_{t_{i-1}}^{t_i}B_h(X_h(t_i)-X_h(s),v_h)ds\\
		&\quad+\int_{t_{i-1}}^{t_i}(F(s,X_h(s))-F(t_{i-1},X_h^{(i-1)}),v_h)_Hds\\
		&\quad+\left(\int_{t_{i-1}}^{t_i}G(s,X_h(s))-G(t_{i-1},X_h^{(i-1)})dL(s),v_h\right)_H.\\
	\end{split}
	\ee
	Testing against $v_h=\phi^{(i)}\in H_h$ in Eq.~\eqref{eq:time_error1} yields
	\be\label{eq:aux3}
	\begin{split}
		&\frac{1}{2}\left(\|\phi^{(i)}\|_H^2-\|\phi^{(i-1)}\|_H^2
		+ \|\phi^{(i)}-\phi^{(i-1)}\|_H^2 \right) 
		+ \gD t B_h(\phi^{(i)},\phi^{(i)}) \\
		&= 
		\int_{t_{i-1}}^{t_i}B_h(X_h(t_i)-X_h(s),\phi^{(i)})ds
		+\left(\int_{t_{i-1}}^{t_i}F(s,X_h(s))-F(t_{i-1},X_h^{(i-1)})ds,\phi^{(i)}\right)_H\\
		&\quad+\left(\int_{t_{i-1}}^{t_i}G(s,X_h(s))-G(t_{i-1},X_h^{(i-1)})dL(s),\phi^{(i)}\right)_H \\
		&=:I+II+III.
	\end{split}
	\ee
	
	We split the first term $I$ via
	\begin{align*}
		I
		&= \int_{t_{i-1}}^{t_i}(X_h(t_i)-X_h(s),a\cdot \nabla \phi^{(i)})_{L^2(\cK_h)}ds
		+ \int_{t_{i-1}}^{t_i}\sum_{E\in\partial\cK_h\setminus\cD^+} 
		\int_E a (X_h(t_i)-X_h(s))^-\cdot\dgj{\phi^{(i)}} dz\, ds\\
		&= \int_{t_{i-1}}^{t_i}(X_h(t_i)-X_h(s),
		a\cdot \nabla \phi^{(i)}-a\cdot \nabla\phi^{(i-1)})_{L^2(\cK_h)} ds\\
		&\quad+\int_{t_{i-1}}^{t_i}(X_h(t_i)-X_h(s),a\cdot \nabla \phi^{(i-1)})_{L^2(\cK_h)}ds\\
		&\quad+ \int_{t_{i-1}}^{t_i}\sum_{E\in\partial\cK_h\setminus\cD^+} 
		\int_E a (X_h(t_i)-X_h(s))^-\cdot\dgj{\phi^{(i)}} dz\, ds\\
		&=:I_a+I_b+I_c.
	\end{align*}
	To bound $I_a$ in expectation, we use the inverse estimate~\eqref{eq:inverse_est}, Young's inequality, and Lemma~\ref{lem:ms-cont-disc} to obtain
	\begin{align*}
		\bE(I_a) 
		&\le C\gD th^{-2}
		\int_{t_{i-1}}^{t_i} \bE(\|X_h(t_i)-X_h(s)\|_H^2)ds
		+ \frac{1}{8}\bE(\|\phi^{(i)}-\phi^{(i-1)}\|_H^2)\\
		&\le C\gD th^{-2}\left( \int_{t_{i-1}}^{t_i} |t_i-s| ds\right) + \frac{1}{8}\bE(\|\phi^{(i)}-\phi^{(i-1)}\|_H^2) \\
		&\le C\gD t h^{-2}\gD t^2 +  \frac{1}{8}\bE(\|\phi^{(i)}-\phi^{(i-1)}\|_H^2).
	\end{align*}
	To bound $I_b$, we use that $X_h(t)-X(s)$ and $\phi^{(i-1)}$ are independent to obtain
	\begin{align*}
		\bE(I_b)
		&= \int_{t_{i-1}}^{t_i}\left(\bE(X_h(t_i)-X_h(s)),\,\bE(a\cdot \nabla \phi^{(i-1)})\right)_H ds\\
		&\le C\int_{t_{i-1}}^{t_i}h^{-1}\|\bE(X_h(t_i)-X_h(s))\|_H
		\bE(\|\phi^{(i-1)}\|_H^2)^{1/2}ds\\
		&\le C\left(h^{-2}\int_{t_{i-1}}^{t_i} |t_i-s|^2 ds 
		+\gD t \bE(\|\phi^{(i-1)}\|_H^2)\right) \\
		&\le C\gD t \left(h^{-2}\gD t^2 + \bE(\|\phi^{(i-1)}\|_H^2)\right),
	\end{align*}
	where we have again used~\eqref{eq:inverse_est} in the second step and Lemma~\ref{lem:ms-cont-disc} for the third line.
	For the term $I_c$, we then also use the splitting $\phi^{(i)}=(\phi^{(i)}-\phi^{(i-1)})+\phi^{(i-1)}$ together with the inverse estimate~\eqref{eq:inverse_est} and Lemma~\ref{lem:ms-cont-disc} to obtain by similar calculations as for $I_a$ and $I_b$ that  
	\begin{align*}
		\bE(I_c)\le 
		C\gD t \left(h^{-2}\gD t^2 + \bE(\|\phi^{(i-1)}\|_H^2)\right)
		+ \frac{1}{8}\bE(\|\phi^{(i)}-\phi^{(i-1)}\|_H^2).
	\end{align*}
	Collecting the estimates for $I_a-I_c$ and using that $h^2\ge \gD t$ then shows 
	\begin{align*}
		\bE(I)&\le
		C\gD t \left(\gD t + \bE(\|\phi^{(i-1)}\|_H^2)\right)
		+ \frac{1}{8}\bE(\|\phi^{(i)}-\phi^{(i-1)}\|_H^2).
	\end{align*}

	The term $II$ is estimated by Young's inequality, Assumption~\ref{ass2}~\ref{item:Lipschitz2} and Lemma~\ref{lem:ms-cont-disc} via
	\bee
	\begin{split}
		\bE(II)
		&=\bE\left(\int_{t_{i-1}}^{t_i}(F(s,X_h(s))-F(t_{i-1},X_h(s)),\phi^{(i)})_Hds\right) \\
		&\quad + \bE\left(\int_{t_{i-1}}^{t_i}(F(t_{i-1},X_h(s))-F(t_{i-1},X_h(t_{i-1})),\phi^{(i)})_Hds\right)\\
		&\quad + \bE\left(\int_{t_{i-1}}^{t_i}(F(t_{i-1},X_h(t_{i-1}))-F(t_{i-1},X_h^{(i-1)}),\phi^{(i)})_Hds\right) \\
		&\le C\int_{t_{i-1}}^{t_i}\bE(\|F(s,X_h(s))-F(t_{i-1},X_h(s))\|_H^2) ds\\
		&\quad+
		C\int_{t_{i-1}}^{t_i}\bE(\|F(t_{i-1},X_h(s))-F(t_{i-1},X_h(t_{i-1}))\|_H^2) ds\\
		&\quad+
		C\int_{t_{i-1}}^{t_i}\bE(\|F(t_{i-1},X_h(t_{i-1}))-F(t_{i-1},X_h^{(i-1)})\|_H^2) ds
		+
		\gD t\bE(\|\phi^{(i)}\|_H^2)\\
		&\le C\left(\gD t^2 + \int_{t_{i-1}}^{t_i} \bE(\|X_h(s)-X_h(t_{i-1})\|_H^2)ds
		+\gD t \bE(\|\phi^{(i)}\|_H^2) \right)\\
		&\le C\gD t \left(\gD t +\bE(\|\phi^{(i)}\|_H^2) \right).
	\end{split}
	\eee
	For the last term $III$, we use independence of $X_h(t)-X_h(s)$ and $\phi^{(i-1)}$ together with the Cauchy-Schwarz/Young's inequality and It\^o's isometry to obtain similarly
	\bee
	\begin{split}
		\bE(III)
		&=\bE\left(\left(
		\int_{t_{i-1}}^{t_i}G(s,X_h(s))-G(t_{i-1},X_h(s))dL(s),
		\phi^{(i)}-\phi^{(i-1)}\right)_H\right) \\
		&\quad + \bE\left(\left(
		\int_{t_{i-1}}^{t_i}G(t_{i-1},X_h(s))-G(t_{i-1},X_h(t_{i-1}))dL(s),
		\phi^{(i)}-\phi^{(i-1)}\right)_H\right)\\
		&\quad + \bE\left(\left(
		\int_{t_{i-1}}^{t_i}G(t_{i-1},X_h(t_{i-1}))-G(t_{i-1},X_h^{(i-1)})dL(s),
		\phi^{(i)}-\phi^{(i-1)}\right)_H\right) \\
		&\le 
		C\int_{t_{i-1}}^{t_i}
		\bE(\|G(s,X_h(s))-G(t_{i-1},X_h(s))\|_{\LHS(\cU, H)}^2) ds\\
		&\quad+
		C\int_{t_{i-1}}^{t_i}
		\bE(\|G(t_{i-1},X_h(s))-G(t_{i-1},X_h(t_{i-1}))\|_{\LHS(\cU, H)}^2) ds\\
		&\quad+
		C\int_{t_{i-1}}^{t_i}
		\bE(\|G(t_{i-1},X_h(t_{i-1}))-G(t_{i-1},X_h^{(i-1)})\|_{\LHS(\cU, H)}^2) ds
		+ \frac{1}{4}\bE(\|\phi^{(i)}-\phi^{(i-1)}\|_H^2)\\
		&\le C\gD t^2 + \frac{1}{4}\bE(\|\phi^{(i)}-\phi^{(i-1)}\|_H^2).
	\end{split}
	\eee
	
	Substituting all estimates in~\eqref{eq:aux3}, regrouping some terms and summing with respect to $i$ thus yields with $\|\phi^{(i-1)}\|_H=0$ that
	\begin{align*}
		\frac{1}{2}\bE(\|\phi^{(i)}\|_H^2) + \gD t \sum_{j=1}^iB_h(\phi^{(j)},\phi^{(j)}) 
		&\le 
		C\gD t\left(
		\gD t  + \sum_{j=1}^i\bE(\|\phi^{(j)}\|_H^2) 
		\right).
	\end{align*}
	The claim then follows by the discrete Grönwall inequality and Lemma~\ref{lem:blf_pos}.
\end{proof}

 \begin{rem}\label{rem:error-balancing}
 	For $\gg\ge 1$, the condition $h^2\ge \gD t$ is satisfied by setting $\gD t = h^{2\gg}$, which naturally yields the overall error equilibration 
	\be\label{eq:time-space-error}
	\sup_{t\in\bT}\|X(t)-\ol X_h(t)\|_{L^2(\gO; H)} 
	\le C \left(\gD t^{1/2} + h^{\min(\gg,2)}\right)
	= C h^{\min(\gg,2)}.
	\ee
	
	For the case that $\gg\in(1/2,1)$ such an error balancing is not straightforward in view of the proof Theorem~\ref{thm:time_error}. However, our numerical results in Section~\ref{sec:numerics} suggest that an overall error of order $\cO(h^\gg)$ may also be recovered when $\gg\in(1/2,1)$ with the parameters $\gD t = h^{2\gg}$.
	On the other hand, a suitable extension of Theorem~\ref{thm:time_error} to the case $\gg<1$ is far from obvious, and requires a more involved analysis. In order to keep the length of this manuscript reasonable, we leave this as a subject for future work. 
\end{rem}

\section{Noise Approximation}\label{sec:noise}
After discretizing the temporal and spatial domain of Problem~\eqref{eq:ste_weak}, it is in general necessary to derive a numerically tractable approximation of the infinite-dimensional driving noise $L$.
For this, we will utilize a series representation of $L$ and truncate the expansion after a finite number of terms.
Since the covariance operator $Q$ of $L$ is symmetric and of trace class, $L$ admits the \textit{\KL expansion}
\be\label{eq:L_KL}
L(t)=\sum_{k\in\bN} (L(t),e_k)_Ue_k,\quad t\in\bT.
\ee
The scalar products $(L(\cdot),e_k)_H$ are one-dimensional \textit{uncorrelated, but not independent}, L\'evy processes with zero mean and variance $\eta_k$ (see \cite{PZ07}).
In general, infinitely many of the eigenvalues $\eta_k$ are strictly greater than zero, hence we truncate the series in Eq.~\eqref{eq:L_KL} after $N\in\bN$ terms to obtain
\bee
L_N(t):=\sum_{k=1}^N (L(t),e_k)_Ue_k,\quad t\in\bT.
\eee
It can be shown, see for example \cite{BS18a}, that $L_N$ converges to $L$ in mean-square uniformly on $\bT$ with the truncation error bounded by
\bee
\bE(\|L_N(t)-L(t)\|_U^2)\le T\sum_{k>N}\eta_k,\quad t\in\bT.
\eee
When simulating $L_N$, it is vital to generate $(L(\cdot),e_1),\dots,(L(\cdot),e_n)$ as uncorrelated, but stochastically dependent L\'evy processes for fixed $N$.
Besides the truncation, another bias may occur when sampling the one-dimensional processes $((L(t),e_k)_H,t\in\bT)$.
For $\eta_k>0$, consider the normalized processes
\be\label{eq:1d_marginal}
\ell_k=(\ell_k(t),t\in\bT):=\frac{((L(t),e_k)_U,t\in\bT)}{\sqrt{\eta_k}},
\ee
with unit variance at $t=1$. Denoting by $\stackrel{\cL}{=}$ equality in distribution,
the identity
\bee
L_N(t)\stackrel{\cL}{=}\sum_{k=1}^N \sqrt{\eta_k}\ell_k(t)e_k
\eee
holds with respect to probability law of $L_N(t)$.
For a general one-dimensional L\'evy process $\ell_k$, it is not possible to sample from the exact distribution of $\ell_k(t)$ for arbitrary $t\in\bT$.
There are a few important exceptions, for instance normal-inverse Gaussian (NIG) or variance Gamma (VG) processes (see \cite{S03}), in any other case, however, one is forced to use approximate simulation algorithms.
The most popular technique is the compound Poisson approximation (CPA), see for instance \cite{AR01,DHP12,F11,R97}, which usually guarantees weak convergence. A drawback of the CPA methods is that it requires rather strong assumptions on the one-dimensional L\'evy processes $\ell_k$ to bound the approximation error in a mean-square sense and is difficult to implement.
Another approach is to use the Fourier inversion (FI) technique introduced in \cite{BS18a}, which ensures a strong error control in $L^p(\gO;\bR)$ under relatively weak assumptions on $\ell_k$. With the FI method, we are able to approximate very general types of L\'evy noise and control the mean-squared error, for instance if $L$ stems from the important class of \textit{generalized hyperbolic} (GH) L\'evy processes introduced in \cite{BN77,BN78}.
To allow for arbitrary approximation techniques, we formulate the following assumption.
\begin{assumption}\label{ass:noise}
	Let $\widetilde\ell_k$ be arbitrary approximations of $\ell_k$ (based on CPA, FI, etc.) such that the processes $(\widetilde\ell_k,k\in\bN)$ are jointly uncorrelated, but stochastically dependent, and let
	\bee
	\widetilde L_N(t):=\sum_{k=1}^N\sqrt{\eta_k}\widetilde\ell_k(t)e_k
	\eee
	be the approximated $U$-valued L\'evy field. There is a constant $\eps_L> 0$ such that for all $k\in\bN$ and $t\in\bT$
	\bee
	\bE(|\widetilde\ell_k(t)-\ell_k(t)|^2)\le \eps_L.
	\eee
\end{assumption}

\begin{rem}
	Assumption~\ref{ass:noise} yields that the overall noise approximation error is bounded by
	\be\label{eq:noise_error}
	\bE(\|L(t)-\widetilde L_N(t)\|_U^2)\le T\sum_{k>N}\eta_k+\eps_L\sum_{k=1}^N\eta_k,\quad t\in\bT,
	\ee
	hence there is a separation between the truncation error with respect to $N$ and the simulation bias $\eps_L$.
	Often, an arbitrary small error $\eps_L$ may be achieved with sufficient computational effort and it is possible to reduce the noise approximation error in Eq.~\eqref{eq:noise_error} to any desired amount by increasing the number of terms in the expansion and decreasing $\eps_L$. This is for instance the case for GH L\'evy fields approximated by FI as in \cite{BS18a}.
	Moreover, we achieve an equilibration between both types of errors in the sense that
	\bee
	\bE(\|L(t)-L_N(t)\|_U^2)=T\sum_{k>N}\eta_k\approx\eps_L\sum_{k=1}^N\eta_k=\bE(\|L_N(t)-\widetilde L_N(t)\|_U^2).
	\eee
\end{rem}

Substituting $L$ by $\widetilde L_N$ in Eq.~\eqref{eq:ste_weak_td} yields the fully discrete problem to find $(\widetilde X^{(i)}_{h,N},i=1,\dots,m)\subset H_h$ such that for all $v_h\in H_h$ and $i=1,\dots,m$ it holds
\be\label{eq:ste_weak_fd}
\begin{split}
	(\widetilde X^{(i)}_{h,N}-\widetilde X_{h,N}^{(i-1)},v_h)_H+\gD tB_h((\widetilde X_{h,N}^{(i)},v_h)&=(F(t_{i-1},\widetilde X	
	_{h,N}^{(i-1)})\gD t,v_h)_H\\
	&\quad+\left(G(t_{i-1},\widetilde X_{h,N}^{(i-1)})\gD \widetilde L_N^{(i)},v_h\right)_H,
\end{split}
\ee
where $\widetilde X^{(0)}_{h,N}:=X_h^{(0)}=\cP_hX_0$ and $\gD\widetilde L_N^{(i)}:=\widetilde L_N(t_i)-\widetilde L_N(t_{i-1})$.
To complete the error analysis, we derive the overall approximation error between $\widetilde X_{h,N}^{(\cdot)}$ and the unbiased weak solution $X$ to Eq.~\eqref{eq:ste_weak}.

\begin{thm}\label{thm:overall_error}
	Let Assumption~\ref{ass2} hold with $\gg:=\min(\gg_0, \gg_F, \gg_G)\ge 1$, let Assumption~\ref{ass:noise} hold, and denote by $X$ and $\widetilde X^{(\cdot)}_{h,N}$ the solutions to Problem~\eqref{eq:ste_weak} and Problem~\eqref{eq:ste_weak_fd}, respectively.
	Then, there is a $C>0$ such that for any $\gD t\in(0,1]$ with $\gD t\le h^2$ and $i=1,\dots, m$ it holds that
	\bee
	\|X(t_i)-\widetilde X_{h,N}^{(i)}\|_{L^2(\gO;H)}
	\le C \left(h^{\min(\gg,2)}+\gD t^{1/2}  +\left(\sum_{k>N}\eta_k+\eps_L\sum_{k=1}^N\eta_k\right)^{1/2}\right).
	\eee
\end{thm}
\begin{proof}
	We define $\widetilde\psi_N^{(i)}:=X_h^{(i)}-\widetilde X_{h,N}^{(i)}\in H_h$ for $i=1,\dots,m$, and obtain by in Eqs.~\eqref{eq:ste_weak_td} and~\eqref{eq:ste_weak_fd} that
	\begin{align*}
		(\widetilde\psi_N^{(i)}-\widetilde\psi_N^{(i-1)},v_h)_H
		+\gD t B_h(\widetilde\psi_N^{(i)},v_h)
		&=\gD t(F(t_{i-1},X_h^{(i-1)})-F(t_{i-1},\widetilde X_{h,N}^{(i-1)}),v_h)_Hds\\
		&\quad+\left(G(t_{i-1},X_h^{(i-1)})\gD L_i-G(t_{i-1}\widetilde X_{h,N}^{(i-1)}) \gD \widetilde L_i,v_h\right)_H\\
		&=\gD t(F(t_{i-1},X_h^{(i-1)})-F(t_{i-1},\widetilde X_{h,N}^{(i-1)}),v_h)_Hds\\
		&\quad+\left(G(t_{i-1},X_h^{(i-1)})\gD L_i
		-G(t_{i-1} X_{h,N}^{(i-1)}) \widetilde \gD L_i,v_h\right)_H\\
		&\quad+\left( G(t_{i-1},X_{h,N}^{(i-1)})\widetilde \gD L_i
		-G(t_{i-1}\widetilde X_{h,N}^{(i-1)})\gD \widetilde L_i,v_h\right)_H.
	\end{align*}
	Testing against $v_h=\widetilde\psi_N^{(i)}$, taking expectations and using that $\widetilde\psi_N^{(i-1)}$ is independent of $\gD L_i$ and $\gD \widetilde L_i$ then yields with Young's inequality that
	\begin{align*}
		&\frac{1}{2}\left(
		\bE(\|\widetilde\psi_N^{(i)}\|_H^2)-\bE(\|\widetilde\psi_N^{(i-1)}\|_H^2)
		+
		\bE(\|\widetilde\psi_N^{(i)}-\widetilde\psi_N^{(i-1)}\|_H^2)
		\right)
		+\gD t B_h(\widetilde\psi_N^{(i)},\widetilde\psi_N^{(i)})\\
		&=\gD t\bE\left((F(t_{i-1},X_h^{(i-1)})-F(t_{i-1},\widetilde X_{h,N}^{(i-1)}),\widetilde\psi_N^{(i)})_H\right)ds\\
		&\quad+\bE\left(\left(G(t_{i-1},X_h^{(i-1)})\gD L_i
		-G(t_{i-1} X_{h,N}^{(i-1)}) \widetilde \gD L_i,\widetilde\psi_N^{(i)}-\widetilde\psi_N^{(i-1)}\right)_H\right)\\
		&\quad+\bE\left(\left( G(t_{i-1},X_{h,N}^{(i-1)})\widetilde \gD L_i
		-G(t_{i-1}\widetilde X_{h,N}^{(i-1)})\gD \widetilde L_i,\widetilde\psi_N^{(i)}-\widetilde\psi_N^{(i-1)}\right)_H\right)\\
		&\le \frac{\gD t}{2} \bE(\|F(t_{i-1},X_h^{(i-1)})-F(t_{i-1},\widetilde X_{h,N}^{(i-1)})\|_H^2) +\frac{\gD t}{2}\bE(\|\widetilde\psi_N^{(i)}\|_H^2)\\
		&\quad+ \bE(\|G(t_{i-1},X_h^{(i-1)})(\gD L_i - \widetilde \gD L_i)\|_H^2) \\
		&\quad+ \bE(\|(G(t_{i-1},X_h^{(i-1)})-G(t_{i-1},\widetilde X_h^{(i-1)}))\widetilde \gD L_i\|_H^2) 
		+\frac{1}{2}\bE(\|\widetilde\psi_N^{(i)}-\widetilde\psi_N^{(i-1)}\|_H^2).
	\end{align*}
	Assumption~\ref{ass:noise} and the It\^o isometry yield for the "middle" term that 
	\begin{align*}
		\bE(\|G(t_{i-1},X_h^{(i-1)})(\gD L_i - \widetilde \gD L_i)\|_H^2) 
		&=
		\bE(\|G(t_{i-1},X_h^{(i-1)})\|_{\LHS(\cU, H)})
		\bE(\|\gD L_i - \widetilde \gD L_i\|_U^2) \\
		&\le C \gD t \left(\sum_{k>N}\eta_k+\eps_L\sum_{k=1}^N\eta_k\right).
	\end{align*}
	By analogous arguments as in the proof of Theorem~\ref{thm:time_error} and the discrete Grönwall inequality, we obtain
	\begin{align*}		
		\bE(\|\widetilde\psi_N^{(i)}\|_H^2)
		\le C \left(\sum_{k>N}\eta_k+\eps_L\sum_{k=1}^N\eta_k\right), 
	\end{align*}
	where $C>0$ is independent of $i, N$ and $\eps_L$.
	The claim then follows as an immediate consequence of Remark~\ref{rem:error-balancing} and the triangle inequality.
\end{proof}

\section{Numerical Experiments}\label{sec:numerics}

For the numerical experiments we consider the spatial domain $\cD=(0,1)$ with time interval $\bT=[0,1]$, take $H=U=L^2((0,1))$, and let $Q$ be the Mat\'ern covariance operator from Examples~\ref{ex:matern} and~\ref{ex:matern2}:
\bee
[Q\phi](x):=\int_0^1 \frac{2^{1-\nu}}{\Gamma(\nu)}\big(\sqrt{2\nu}\frac{|x-y|}{\rho}\big)^\nu K_\nu\big(\sqrt{2\nu}\frac{|x-y|}{\rho}\big)\phi(y)dy,\quad \phi\in U,\; x\in(0,1).
\eee
We fix the correlation length to $\rho=1/4$ and vary the smoothness parameter $\nu>0$ throughout our experiments. The eigenpairs $((\eta_k,e_k),k\in\bN)$ of $Q$ may be approximated by solving a discrete eigenvalue problem and interpolation, see \cite[Chapter 4.3]{RW06}.
Provided that $\nu>1/2$, Remark~\ref{rem:ass2} shows that Assumption~\ref{ass2}\ref{item:lingrowthG} is satisfied for $\gg_G<\nu$ and $\gb<\nu/(1+2\nu)$.
We consider GH L\'evy fields, i.e. the one-dimensional processes $(\ell_i,i\in\bN)$ from Eq.~\ref{eq:1d_marginal} are uncorrelated GH L\'evy processes. More importantly, for each $N\in\bN$ the vector-valued process $(\text{GH}_N(t),t\in\bT):=((\ell_1(t),\dots,\ell_N(t)),t\in\bT)$ is a $N$-dimensional GH L\'evy process with parameters $\widehat\gl\in\bR,\widehat\ga>0,\widehat\gd>0,\widehat\vartheta\in\bR^n,\widehat\mu\in\bR^N$ and $\widehat\gG\in\bR^{N\times N}$, where $\widehat\ga^2>\widehat\vartheta\cdot\widehat\gG\widehat\vartheta$ and the matrix $\widehat\gG$ is symmetric, positive definite with unit variance. The characteristic function of $\text{GH}_N$ is given for $u\in\bR^N$ by
\begin{equation*}
	\bE(e^{iu\cdot \text{GH}_N(t)})=e^{iu\cdot\widehat\mu t}
	\left(\frac{\widehat\ga^2-\widehat\vartheta\cdot\widehat\gG\widehat\vartheta}{\widehat\ga^2-(iu+\widehat\vartheta)\cdot\widehat\gG(iu+\vartheta)}\right)^{\widehat\gl t/2}
	\frac{K_{\widehat\gl}(\widehat\gd(\widehat\ga^2-(iu+\widehat\vartheta)\cdot\widehat\gG(iu+\widehat\vartheta))^{1/2})^t}
	{K_{\widehat\gl}(\widehat\gd(\widehat\ga^2-\widehat\vartheta\cdot\widehat\gG\widehat\vartheta)^{1/2})^t}.
\end{equation*}
We achieve a zero-mean process by setting $\widehat\vartheta=\widehat\mu=(0,\dots,0)$. An important class of the GH family are \textit{normal inverse Gaussian} (NIG) processes, where $\widehat\gl=-1/2$.
For more details on multidimensional GH distributions and the simulation of GH L\'evy fields we refer again to \cite{BS18a} and the references therein.
In all subsequent experiments, we use a NIG L\'evy field with $\widehat\ga=10,\,\widehat\gd=1,\,\widehat\vartheta=\widehat\mu=(0,\dots,0)$ and \linebreak $\widehat\gG=\1_N$ for each truncation index $N$.
The choice of NIG fields is motivated by the results from \cite{AKW10}, where the authors pointed out that this class of L\'evy fields is well-suited to fit empirical log-returns in  electricity forward markets.
We are able to simulate multidimensional NIG processes without bias, i.e., Assumption~\ref{ass:noise} holds with $\eps_L=0$.

As stochastic transport problem we consider a slight modification of the energy forward model from \cite{BB14} (to include multiplicative noise) given by
\be\label{eq:forward}
dX(t,x)=\partial_x X(t,x)+\Sigma(X(t,x),x)^2dt +\Sigma(X(t,x),x)dL(t,x),\quad x\in\cD,\, t\in\bT.
\ee
For positive parameters $\ga, \sigma>0$  we use the coefficient function
\bee
\Sigma:H\times\ol \cD\to\bR_+^0,\quad  (X,x)\mapsto \sigma (e^{-\ga x}-e^{-\ga}) X,
\eee
and matching initial/inflow boundary conditions given by
\bee
X_0(x)=e^{-\ga x}+\frac{\sigma^2K_0(\widehat\ga)}{\ga\pi}(1-e^{-\ga x}),\quad X(t,1)=e^{-\ga}.
\eee
Note that Eq.~\eqref{eq:forward} may be transformed to a problem with homogeneous boundary conditions by replacing $\Sigma$ and $X_0$ by
\begin{equation*}
	\Sigma^{hom}(X,x):= \Sigma(X+e^{-\ga},x)\quad\text{and}\quad
	X_0^{hom}(x):= X_0(x)-e^{-\ga},
\end{equation*}
see Remark~\ref{rem:BC}.
We fix the values $\ga=0.5$ and $\sigma=1$ for our experiments.
The coefficients in Eq.~\eqref{eq:forward} are time-independent and the relation $F(t,X)=\Sigma(X(t,\cdot),\cdot)^2=G(t,X)^2$ is imposed to ensure the absence of arbitrage in the market, see \cite{BB14}.
This entails that $F$ is only locally Lipschitz with respect to $X$, while Assumption~\ref{ass2} still holds for $G$. Nevertheless, the quadratic growth of $F$ did not cause any problems in our experiments, and hence we stick to this particular example.
As $F, G$ and $X_0$ vanish near the outflow boundary for $x=1$, but their derivatives do not,  Assumption~\ref{ass2} is satisfied with $\gg_0=\gg_F=3/2-\eps$ and $\gg_G=\min(3/2-\eps,\nu)$ for all $\eps>0$.
Samples of $X$ are given for $\nu=1$ and $\nu=3$ in Figure~\ref{fig:ForwardSample}.
\begin{figure}[h]
	\centering
	\subfigure{\includegraphics[scale=0.36]{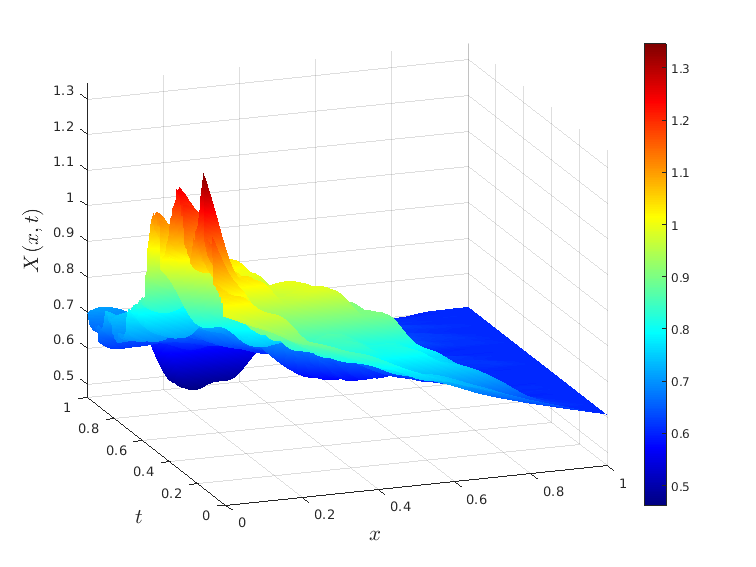}}
	\subfigure{\includegraphics[scale=0.36]{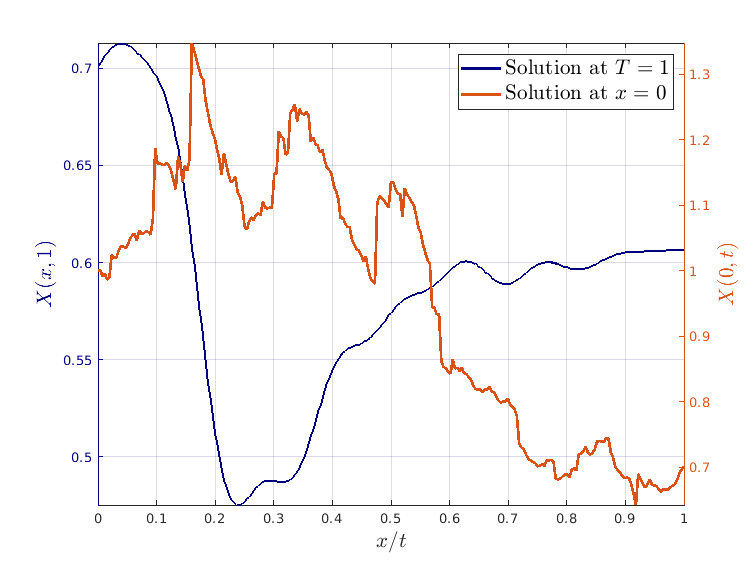}}
	\subfigure{\includegraphics[scale=0.36]{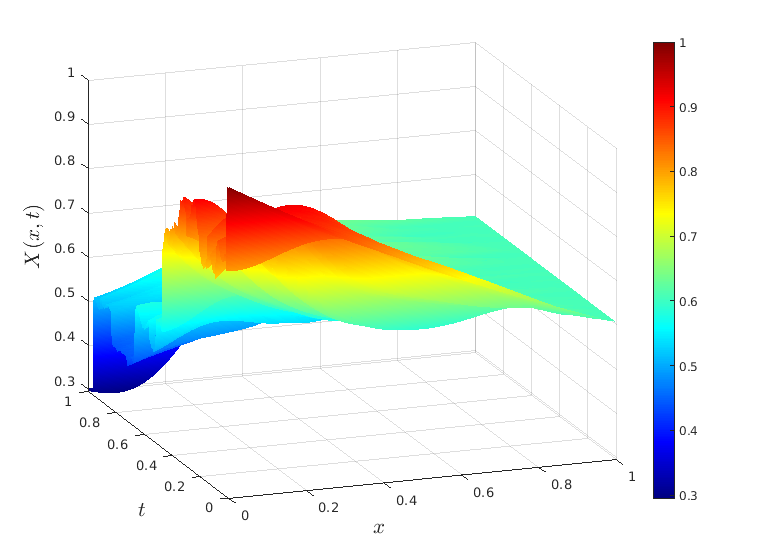}}
	\subfigure{\includegraphics[scale=0.36]{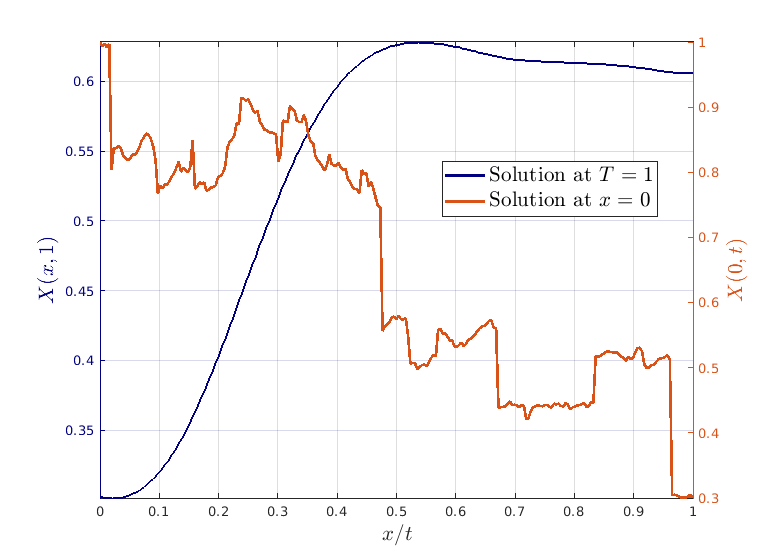}}
	\caption{Left column: samples of the solution to the forward model. Right column: plots of the surface at the outflow boundary $x=1$ (blue) and for $T=1$ (orange). The smoothness parameters of the covariance function are $\nu=1$ in the top row and $\nu=3$ in the bottom row.}
	\label{fig:ForwardSample}
\end{figure}

We use the fully discrete scheme~\eqref{eq:ste_weak_fd} to discretize Eq.~\eqref{eq:forward}. The space $H_h$ consists of the piecewise linear DG functions with respect to an equidistant refinement of $\cD=(0,1)$ with mesh width $h>0$. 
Recalling that $\gg<2$ and $\eps_L=0$, Theorem~\ref{thm:overall_error} predicts for $\gg\ge 1$ an overall error of 
\bee
\bE(\|X(T)-\widetilde X_{h,N}^{(m)}\|_H^2)
\le C \left(h^{2\gg}+\gD t + \sum_{k>N}\eta_k\right).
\eee

We consider $\nu\in\{0.5,1,1.5,2,2.5,3\}$, assume for simplicity that $\gg=\min(3/2,\nu)$, 
and use spatial refinements of $h=2^{-3},\dots,2^{-7}$ for any given $\nu$. 
Based on Remark~\ref{rem:error-balancing}, we determine $\gD t$ and $N$ such that
\be\label{eq:errorequi}
\gD t = \sum_{i>N}\eta_k = \max(h^{2\gg},2^{-20}),
\ee
where the choice $\gD t = 2^{-20}$ only applies in the case $\nu\ge 1.5$ for the refinement $h=2^{-7}$.
We approximate $X(T)$ by a reference solution $X_{ref}(T)$ that is generated with $h_{ref}=2^{-9}$ and $\gD t$ and $N$ according to Eq.~\eqref{eq:errorequi}.
The overall root-mean-squared error (RMSE) from Theorem~\ref{thm:overall_error} is estimated by averaging 200 independent samples of $X_{ref}(T)-\widetilde X^{(m)}_{h,N}$, i.e.,
\bee
\bE(\|X(T)-\widetilde X^{(m)}_{h,N}\|_H^2)\approx \frac{1}{200}\sum_{l=1}^{200}\|(X_{ref}(T)-\widetilde X^{(m)}_{h,N})_l\|_H^2,
\eee
where the subscript $l$ denotes the $l$-th Monte Carlo sample. The same realization of the L\'evy noise $L$ is used for $X_{ref}$ and $\widetilde X^{(m)}_{h,N}$ in any of the 200 samples to estimate the pathwise, strong convergence of the algorithm.
By Eq.~\eqref{eq:errorequi}, we have for $\gg\ge 1$ that 
\begin{equation*}
	\bE(\|X(T)-\widetilde X^{(m)}_{h,N}\|_H^2)\approx C h^{2\gg}
	\quad\text{and}\quad
	\log\left(\bE(\|X(T)-\widetilde X^{(m)}_{h,N}\|_H^2)^{1/2}\right) \approx \gg \log(h) +\log(C).
\end{equation*}
Hence, we perform a linear regression of the estimated log-RMSE on the log-refinement to obtain an empirical estimate of $\gg$ to compare with our theoretical findings.

We display the results for $\nu\in\{0.5,1,1.5,2,2.5,3\}$ and $h=2^{-3},\dots,2^{-7}$ in Figure~\ref{fig:ForwardConvergence}. As expected, a larger value of $\nu$ increases smoothness, and therefore causes a faster error decay with respect to $h$. This effect saturates around $\nu=2$, as the smoothness of the problem is not anymore limited by the noise, but by the "kink" in the solution at the inflow boundary.
Moreover, the estimated empirical convergence rates in the right plot of Figure~\ref{fig:ForwardConvergence} are in line with our findings from Theorem~\ref{thm:DGerror} and the spectral analysis of the Mat\'ern kernel in Example~\ref{ex:matern2}. The convergence rate is $\gg\approx\nu$ until its saturation point at $\nu=2$, where it remains at $\gg\approx 1.5$, since the solution is at most $H^{3/2-\eps}(\cD)$-regular for $\nu>1.5$. Finally, we remark that the discretization scheme also has an error decay with rate $\gg_G\approx 0.5$ for the borderline case $\nu=0.5$ with $\gD t=h$, which is not covered by Theorem~\ref{thm:time_error}.
\begin{figure}[h]
	\centering
	\subfigure{\includegraphics[scale=0.36]{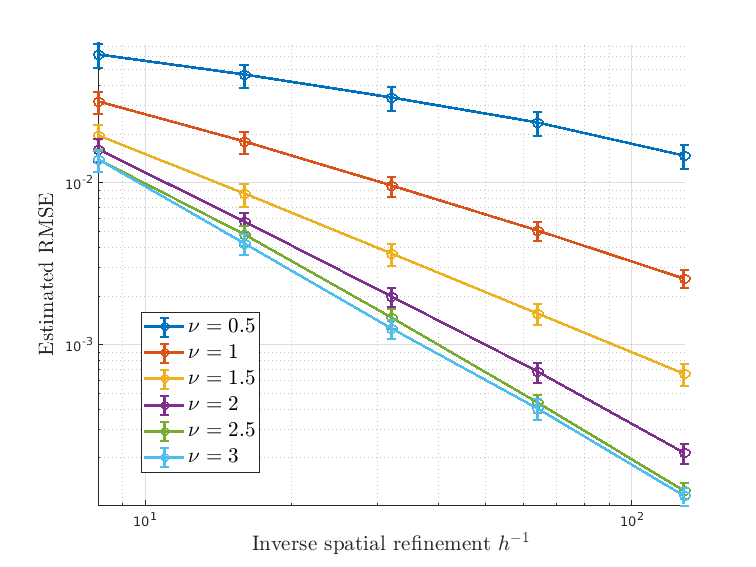}}
	\subfigure{\includegraphics[scale=0.36]{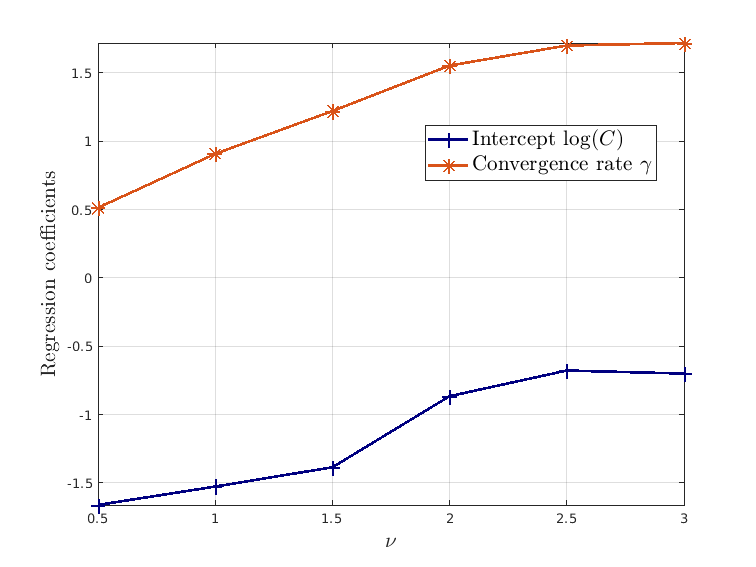}}
	\caption{Left: RMSE vs. inverse spatial refinement $h^{-1}$, the bars on each RMSE curve indicate the 95-\% confidence interval of the estimated error. Right: estimated convergence rates of the Backward Euler -- Petrov-DG scheme for Eq.~\eqref{eq:forward}.}
	\label{fig:ForwardConvergence}
\end{figure}


\section*{Acknowledgements}

This work is partially funded by Deutsche Forschungsgemeinschaft (DFG, German Research Foundation) under Germany's Excellence Strategy - EXC 2075 - 39740016 and it is greatly appreciated. AS was partly funded by ETH Foundations of Data Science (ETH-FDS). We thank Prof. Dr. Christoph Schwab for insightful discussions that led to a significant improvement of the manuscript. 

\addcontentsline{toc}{section}{References}

\small

\bibliographystyle{abbrv}
\bibliography{LevySPDEs}   


\end{document}